\documentclass[11pt,leqno]{article}
\usepackage{amsfonts}
\usepackage{amsmath, amsthm, amsfonts, amssymb, color}
\usepackage{mathrsfs}
\renewcommand{\bar}{\overline}
\renewcommand{\hat}{\widehat}
\renewcommand{\tilde}{\widetilde}

\newtheorem{thm}{Theorem}[section]

\newtheorem{lem}[thm]{Lemma}

\newtheorem{exa}[thm]{Example}
\theoremstyle{definition}

\newcommand{\scr}[1]{\mathscr #1}
\definecolor{wco}{rgb}{0.5,0.2,0.3}

\numberwithin{equation}{section} \theoremstyle{remark}
\newtheorem{rem}{Remark}[section]

\newcommand{\ua}{\uparrow}

\title{{\bf Least squares estimation for path-distribution dependent stochastic differential equations}}

\author{
{\bf  Panpan Ren and Jiang-Lun Wu}\\
\footnotesize{Department of Mathematics, Swansea University,
Singleton Park, Swansea SA2 8PP, UK} \\
E-mails: 673788@swansea.ac.uk \,\, j.l.wu@swansea.ac.uk}

\begin{document}
\def\R{\mathbb R}  \def\ff{\frac} \def\ss{\sqrt} \def\B{\mathbf
B}
\def\N{\mathbb N} \def\kk{\kappa} \def\m{{\bf m}}
\def\dd{\delta} \def\DD{\Delta} \def\vv{\varepsilon} \def\rr{\rho}
\def\<{\langle} \def\>{\rangle} \def\GG{\Gamma} \def\gg{\gamma}
  \def\nn{\nabla} \def\pp{\partial} \def\EE{\scr E}
\def\d{\text{\rm{d}}} \def\bb{\beta} \def\aa{\alpha} \def\D{\scr D}
  \def\si{\sigma} \def\ess{\text{\rm{ess}}}
\def\beg{\begin} \def\beq{\begin{equation}}  \def\F{\scr F}
\def\Ric{\text{\rm{Ric}}} \def\Hess{\text{\rm{Hess}}}
\def\e{\text{\rm{e}}} \def\ua{\underline a} \def\OO{\Omega}  \def\oo{\omega}
 \def\tt{\tilde} \def\Ric{\text{\rm{Ric}}}
\def\cut{\text{\rm{cut}}} \def\P{\mathbb P} \def\ifn{I_n(f^{\bigotimes n})}
\def\C{\scr C}      \def\aaa{\mathbf{r}}     \def\r{r}
\def\gap{\text{\rm{gap}}} \def\prr{\pi_{{\bf m},\varrho}}  \def\r{\mathbf r}
\def\Z{\mathbb Z} \def\vrr{\varrho} \def\ll{\lambda}
\def\L{\scr L}\def\Tt{\tt} \def\TT{\tt}\def\II{\mathbb I}
\def\i{{\rm in}}\def\Sect{{\rm Sect}}\def\E{\mathbb E} \def\H{\mathbb H}
\def\M{\scr M}\def\Q{\mathbb Q} \def\texto{\text{o}} \def\LL{\Lambda}
\def\Rank{{\rm Rank}} \def\B{\scr B} \def\i{{\rm i}} \def\HR{\hat{\R}^d}
\def\to{\rightarrow}\def\l{\ell}
\def\8{\infty}\def\X{\mathbb{X}}\def\3{\triangle}
\def\V{\mathbb{V}}\def\M{\mathbb{M}}\def\W{\mathbb{W}}\def\Y{\mathbb{Y}}\def\1{\lesssim}

\def\La{\Lambda}\def\S{\mathbf{S}}

\renewcommand{\bar}{\overline}
\renewcommand{\hat}{\widehat}
\renewcommand{\tilde}{\widetilde}

\maketitle

\begin{abstract}
We study a least squares estimator for an unknown parameter in the
drift coefficient of a path-distribution dependent stochastic
differential equation involving a small dispersion parameter
$\vv>0$. The estimator, based on $n$ (where $n\in\mathbb{N}$) discrete time
observations of the stochastic differential equation, is shown to be
convergent weakly to the true value as $\vv\to0$ and $n\to\infty$.
This indicates that the least squares estimator obtained is
consistent with the true value. Moreover, we obtain the rate of
convergence and derive the asymptotic distribution of least squares
estimator.
\end{abstract}
 \vskip 2cm

\section{Introduction}
Nowadays, stochastic differential equations (SDEs) are widely used
in modelling time evolution of dynamical systems influenced by
random noise, see, e.g., the monographs
\cite{F98,IkedaW,Oksendal,Protter} (and references therein).
Usually, there exist unknown parameters in such modelled systems,
such as those stochastic models with comparably easier structured
stochastic differential equations involving unknown quantities
(see,.e.g., \cite[P.2-4]{B08}). Fundamental issues are then to
estimate certain parameters (i.e., deterministic quantities)
appearing in the stochastic models by certain observations (or by
experimental data). Viewing the drift part of the SDEs as the
averaging evolution of the systems, estimating the drift parameter
of SDEs is hence an important topic. To approach the true value of
the unknown parameter, the asymptotic approach to statistical
estimation is frequently taken an advantage due to its general
applicability and relative simplicity (cf. \cite{B08}). As we know,
the estimations upon the unknown quantities are based generally on
continuous-time or discrete-time observations. Whereas, the
parameter estimation relied on continuous-time observations is a
mathematical idealisation although there is a vast literature
concerned with such topic. On the other hand, no measuring device
can follow continuously  the sample paths of the diffusion processes
involved, which are indeed rather tricky. Whence, in practice the
investigation on the parameter estimations with the help of
discrete-time observations has been received much more attention
recently. Most importantly, the parameter estimation by the aid of
discrete-time observations can be implemented conveniently with a
powerful theory of simulation schemes and numerical analysis of
diffusion processes.

So far, there are numerous methods to investigate the parameter
estimations on the unknown parameters in the drift coefficents; see,
e.g., \cite{Ku04,LS01,P199,SY06} by maximum likelihood estimator
(MLE for short), \cite{D76,Ka88,Ku04,P199} via least squares
estimator (LSE for abbreviation), and \cite{M05} through
trajectory-fitting estimator, to name a few. Diffusion processes
with small noises have been applied considerably in mathematical
finance, see, e.g., \cite{Ku01,Ta04,Y1992b} and references within.
In the past forty years, the asymptotic theory on parameter
estimations for diffusion processes with small noises has also been
developed very well, see, for instance, \cite{GS, Li16,SM03,U04,U08}
for SDEs driven by L\'{e}vy processes with arbitrary moments, and
\cite{HL,LMS,LSS} for SDEs driven by $\alpha$-stable L\'{e}vy noises
which enjoy heavy tail properties.

Recently, from the stochastic modelling perspective and diverse
demanding in practical problems, there has been increasing interest
on studying stochastic differential equations with path-distribution
coefficients, see e.g. \cite{Huang,HRW,W16}  (and references
therein). The distribution-dependent SDEs are also named as
McKean-Vlasov SDEs or mean-filed SDEs, which have been studied
intensively in the literature, see e.g. \cite{DST, LMb} and
references therein. Such kind of SDEs has been applied successfully
in stochastic differential games and stochastic optimal
optimisation, see, e.g., \cite{LM} and references within. Although
McKean-Vlasov SDEs have been applied diffusively in different
research areas, so far there is little work  on parameter
estimations except the existing literature \cite{Wen}, to the best
of our knowledge. In the present paper, we are concerned with the
LSE problem for the path-distribution stochastic differential
equations with small dispersion noise and involving unknown
parameter in the drift. Our key start point is the associated
discrete-time observations of path-distribution dependent SDEs (see
\eqref{eq1} below). We then investigate parameter estimation for
McKean-Vlasov SDEs which are not only path-dependent but also
dependent on the law of the path. Since the state space of the
window process is infinite dimensional, some new procedures need to
be put forward. We succeeded the task by carefully construct the
Euler-Maruyama (EM) discretion scheme of our path-distribution
dependent SDEs. It is also interesting to consider other type
estimations for such equations and we will study them in another
paper.

The rest of the paper is arranged as follows. In Section
\ref{sec0}, we introduce some notation, present the framework of our
paper, and construct the LSE; Section \ref{sec2} is devoted to the
consistency of LSE; Section \ref{sec3} focus on the asymptotic
distribution of LSE. Throughout this paper, we emphasise that $c>0$
is a generic constant which may change from line to line.

\section{Preliminaries}\label{sec0}
We start with some notation and terminology which will be used
later. For  $d,m\in\mathbb{N}$, the set of all positive integers,
 let $(\R^d,\<\cdot,\cdot\>,|\cdot|)$ be the
$d$-dimensional Euclinean space with the inner product
$\<\cdot,\cdot\>$ induced the norm $|\cdot|$ and $\R^d\otimes\R^m$
the collection of all $d\times m$ matrixes with real entries, which
is endowed with the Hilbert-Schmidt norm $\|\cdot\|$. ${\bf
0}\in\R^d$ denotes the zero vector. For a matrix $A$,  $A^*$ denotes
the transpose of $A.$ Concerning a square matrix $A$,  $A^{-1}$
means the inverse  of $A$ provided that $ \mbox{det} A\neq0$. For $
p\in\mathbb{N}$, let $\Theta$ be an open bounded convex subset of
$\R^p$, and $\bar\Theta$ the closure of $\Theta.$ For $r>0$ and
$x\in\R^p$, $B_r(x)$ represents  the closed ball centered at $x$
with the radius $r.$ For $z\in\R^d$, $\dd_z$ denotes Dirac's delta
measure or unit mass at the point $z.$ For a real number $a>0$,
$\lfloor a\rfloor$ stands for the integer part of $a.$ For a random
variable $\xi$, $\mathscr{L}_\xi$ denotes its law. For a fixed
finite  number $r_0>0$,
 $\C:=C([-r_0,0];\R^d)$ means the family of all continuous functions
$f:[-r_0,0]\rightarrow\R^d$, which is  a Polish  (i.e.,   separable,
complete  metric) space under the uniform norm
$\|f\|_\8:=\sup_{-r_0\le\theta\le0}|f(\theta)|$. Generally speaking,
$r_0>0$ is named as the length of memory. For a continuous map
$f:[-r_0,\8)\rightarrow\R^d$ and $t\ge0$, let $f_t\in\C$ be such
that $f_t(\theta)=f(t+\theta)$ for $ \theta\in[-r_0,0]$. In general,
$(f_t)_{t\ge0}$ is called the window (or segment) process of
$(f(t))_{t\ge-r_0}$. $ \mathcal {P}_2(\C)$ stands for the space of
all probability measures on $\C$ with the finite second-order
moment, i.e.,
$\mu(\|\cdot\|_\8^2):=\int_\C\|\zeta\|_\8^2\mu(\d\zeta)<\8$ for
$\mu\in \mathcal {P}_2(\C)$. Define the Wasserstein distance
$\mathbb{W}_2$ on $\mathcal {P}_2(\C)$ by
\begin{equation*}
\mathbb{W}_2(\mu,\nu)=\inf_{\pi\in\mathcal
{C}(\mu,\nu)}\Big(\int_\C\int_\C\|\zeta_1-\zeta_2\|_\8^2\pi(\d
\zeta_1,\d \zeta_2)\Big)^{1/2},~~~~\mu,\nu\in\mathcal {P}_2(\C),
\end{equation*}
where $\mathcal {C}(\mu,\nu)$ signifies the collection of all
probability measures on $\C\times\C$ with marginals $\mu$ and $\nu$
(i.e., $\pi\in\mathcal {C}(\mu,\nu)$ such that
$\pi(\cdot,\C)=\mu(\cdot)$ and $\pi(\C,\cdot)=\nu(\cdot)$),
respectively. Under the distance $\mathbb{W}_2$, $\mathcal
{P}_2(\C)$ is a Polish space; see, e.g., \cite[Lemma 5.3 \& Theorem
5.4]{Chen}. Let $(B(t))_{t\ge0}$ be an $m$-dimensional Brownian
motion defined on the probability space $(\OO,\F,\P)$ with the
filtration $(\F_t)_{t\ge0}$ satisfying the usual condition (i.e.,
$\F_0$ contains all $\P$-null sets and
$\F_t=\F_{t+}:=\bigcap_{s>t}\F_s$).

Through  all the  paper, we fix the time horizon $T>0.$ For  the
scale parameter $\vv\in(0,1)$, we consider a path-distribution
dependent SDE on $(\R^d, \<\cdot,\cdot\>,|\cdot|)$ in the form
\begin{equation}\label{eq1}
\d X^\vv(t)=b(X_t^\vv,\mathscr{L}_{X_t^\vv},\theta)\d
t+\vv\,\si(X_t^\vv,\mathscr{L}_{X_t^\vv})\d B(t),
~~~t\in(0,T],~~~~X_0^\vv=\xi\in\C,
\end{equation}
where $b:\C\times\mathcal {P}_2(\C)\times\Theta\rightarrow\R^d$ and
$\si:\C\times\mathcal {P}_2(\C)\rightarrow\R^d\times\R^m$. In
\eqref{eq1}, we assume that the drift $b$ and the diffusion $\si$
are known apart from the parameter $\theta\in\Theta$ and we
stipulate that  $\theta_0\in\Theta$ is the true value of
$\theta\in\Theta.$

For   any $\zeta_1,\zeta_2\in\C$ and $\mu,\nu\in\mathcal {P}_2(\C)$,
we assume that
\begin{enumerate}
\item[({\bf A1})]  There exist
$\aa_1,\aa_2,\bb_1,\bb_2>0$ such that
\begin{equation*}
\sup_{\theta\in\bar\Theta}|b(\zeta_1,\mu,\theta)-b(\zeta_2,\nu,\theta)|^2\le\aa_1\|\zeta_1-\zeta_2\|_\8^2+\aa_2\mathbb{W}_2(\mu,\nu)^2,
\end{equation*}
and
\begin{equation*}
\|\si(\zeta_1,\mu)-\si(\zeta_2,\nu)\|^2\le
\bb_1\|\zeta_1-\zeta_2\|_\8^2+\bb_2\mathbb{W}_2(\mu,\nu)^2;
\end{equation*}

\item[({\bf A2})]For each random variable $\zeta\in\C$ with $\mathscr{L}_\zeta\in\mathcal
{P}_2(\C)$, $(\si\si^*)(\zeta,\mathscr{L}_\zeta)$ is invertible, and
there exists an $L_1>0$ such that
\begin{equation*}
\|(\si\si^*)^{-1}(\zeta_1,\mu)-(\si\si^*)^{-1}(\zeta_2,\nu)\|\le
L_1\Big\{\|\zeta_1-\zeta_2\|_\8+\mathbb{W}_2(\mu,\nu)\Big\};
\end{equation*}

\item[({\bf A3})] For the initial value $X^\vv_0=\xi$, there exists an $L_2>0$ such that
\begin{equation*}
|\xi(t)-\xi(s)|\le L_2|t-s|,~~~t,s\in[-r_0,0].
\end{equation*}

\end{enumerate}

We further assume that
\begin{enumerate}
\item[(\bf B1)] There exists $K_1>0$ such that
\begin{equation*}
\sup_{\theta\in\bar\Theta}\|(\nn_\theta
b)(\zeta_1,\mu,\theta)-(\nn_\theta b)(\zeta_2,\nu,\theta)\|\le
K_1\Big\{\|\zeta_1-\zeta_2\|_\8+\mathbb{W}_2(\mu,\nu)\Big\},
\end{equation*}
where $(\nn_\theta b)$ means the gradient operator w.r.t. the third
spatial variable.

\item[(\bf B2)]
There exists $K_2>0$ such that
\begin{equation*}
\sup_{\theta\in\bar\Theta}\|(\nn_\theta(\nn_\theta
b^*))(\zeta_1,\mu,\theta)-(\nn_\theta(\nn_\theta
b^*))(\zeta_2,\nu,\theta)\|\le
K_2\Big\{\|\zeta_1-\zeta_2\|_\8+\mathbb{W}_2(\mu,\nu)\Big\}.
\end{equation*}

\end{enumerate}

%
%

%

Before we move forward,  let's give some remarks. 
Under ({\bf A1}), \eqref{eq1} admits   a unique strong solution
$(X^\vv(t))_{t\in[-r_0,T]}$; see, for instance, \cite[Theorem
3.1]{HRW}. For more details on existence and uniqueness of strong
solutions to distribution-dependent SDEs, we would like to refer to,
e.g.,  \cite{DST,MV,W16} and references within. As far as existence
and uniqueness of weak solutions are concerned, please consult,
e.g., \cite{JM,LMb,Ve} for reference. ({\bf B1}) and ({\bf B2}) are
imposed merely to discuss the asymptotic distribution of LSE
constructed below; see Theorem \ref{th2}. ({\bf A3}) is put to
analyze continuity of the window process  associated with
\eqref{eq2}; see Lemma \ref{le5} for more details. Obviously, ({\bf
A2}) holds provided that
$\si(\cdot,\cdot)\equiv\si\in\R^d\otimes\R^m$, a constant matrix,
such that $\si\si^*$ is invertible. Moreover, for the scalar setting
of \eqref{eq1}, ({\bf A2}) is also true in  case of
$\si(x,\mu)=1+|x|$ for any $x\in\R$ and $\mu\in\mathcal {P}_2(\R)$.

Without loss of generality, we assume the stepsize
$\dd=\frac{T}{n}=\ff{r_0}{M}$ for some integers $n,M\in\mathbb{N}$
sufficiently large. Suppose that the solution process
$(X^\vv(t))_{t\in[-r_0,T]}$ is observed at regularly spaced time
points $t_k=k\delta$ for $k=0,1,\cdots,n$. In
this paper, our goal is  to investigate the LSE 
on the parameter $\theta\in\Theta$ based on the sampling data
$(X^\vv(t_k)^n_{k=0}$ with small dispersion $\vv$ and large sample
size $n$ (i.e., small step size $\dd$).

The discrete-time Euler-Maruyama (EM) scheme corresponding to
\eqref{eq1} admits the form
\begin{equation}\label{q3}
Y^\vv(t_k)=Y^\vv(t_{k-1})+b(\hat Y^\vv_{t_{k-1}},\mathscr{L}_{\hat
Y^\vv_{t_{k-1}}},\theta)\dd+\vv\,\si(\hat
Y^\vv_{t_{k-1}},\mathscr{L}_{\hat Y^\vv_{t_{k-1}}})\triangle
B_k,~~~k\ge1,
\end{equation}
and $Y^\vv(t)=X^\vv(t)=\xi(t), t\in[-r_0,0].$ In \eqref{q3}, $\hat
Y_{k\dd}^\vv=\{ \hat Y_{k\dd}^\vv(s):-r_0\le s\le0\}$ is a
$\C$-valued random variable defined as follows: for any
$s\in[-(i+1)\dd,-i\dd]$, $i=1,\cdots,M-1$,
\begin{equation}\label{w2}
\hat
Y_{k\dd}^\vv(s)=Y^\vv((k-i)\dd)+\ff{s+i\dd}{\dd}\{Y^\vv((k-i)\dd)
-Y^\vv((k-i-1)\dd) \},~~~
\end{equation}
i.e., $\hat Y_{k\dd}^\vv$ is the linear interpolation of
$Y^\vv((k-M)\dd)$,
$Y^\vv((k-(M-1))\dd),\cdots,Y^\vv((k-1)\dd),Y^\vv(k\dd)$, and
$\triangle B_k:=B(k\dd)-B((k-1)\dd)$, the increment of Brownian
motion. Motivated by \cite{LMS,LSS,SM03}, for our present setting
 we construct the following contrast function
\begin{equation}\label{eq2}
\Psi_{n,\vv}(\theta)=\vv^{-2}\delta^{-1}\sum_{k=1}^nP_k^*(\theta)\LL_{k-1}^{-1}P_k(\theta),
\end{equation}
where
\begin{equation}\label{w1}
P_k(\theta):=Y^\vv(t_k)-Y^\vv(t_{k-1})-b(\hat
Y_{t_{k-1}}^\vv,\mathscr{L}_{\hat Y_{t_{k-1}}^\vv},\theta)\dd
~~~\mbox{ and }~~~~\LL_k:=(\si\si^*)(\hat
Y_{t_k}^\vv,\mathscr{L}_{\hat Y_{t_k}^\vv})
\end{equation}
for $k=1,\cdots,n$. 
To achieve the LSE of $\theta\in\Theta$, it   suffices to choose an
argument $\hat\theta_{n,\vv}\in\Theta$ such that
\begin{equation}\label{eq3}
\Psi_{n,\vv}(\hat\theta_{n,\vv})=\min_{\theta\in\Theta}\Psi_{n,\vv}(\theta).
\end{equation}
Next, we write $\hat\theta_{n,\vv}\in\Theta$  satisfying \eqref{eq3}
by
\begin{equation*}
\hat\theta_{n,\vv}=\arg\min_{\theta\in\Theta}\Psi_{n,\vv}(\theta).
\end{equation*}
Set
\begin{equation*}
\Phi_{n,\vv}(\theta):=\vv^2(\Psi_{n,\vv}(\theta)-\Psi_{n,\vv}(\theta_0)).
\end{equation*}
 It follows from \eqref{eq3} that
\begin{equation}\label{eq4}
\Phi_{n,\vv}(\hat\theta_{n,\vv})=\min_{\theta\in\Theta}\Phi_{n,\vv}(\theta).
\end{equation}
Likewise, we reformulate $\hat\theta_{n,\vv}\in\Theta$ ensuring
\eqref{eq4} to hold true as
\begin{equation}\label{eq5}
\hat\theta_{n,\vv}=\arg\min_{\theta\in\Theta}\Phi_{n,\vv}(\theta).
\end{equation}
Through the whole paper, $\hat\theta_{n,\vv}$ such that \eqref{eq5}
holds  is named as the LSE of $\theta\in\Theta$.

Before we end this section, we give some remarks.

\begin{rem}
{\rm If $\si(\cdot,\cdot)\in\R^d\otimes\R^d$ is invertible,
\eqref{eq3} can be rewritten as
\begin{equation*}
\ff{\triangle B_k}{\ss\dd}=\ff{1}{\vv^{-1}\ss\dd}\si^{-1}(\hat
Y^\vv_{t_{k-1}}, \mathscr{L}_{\hat Y^\vv_{t_{k-1}}})P_k(\theta).
\end{equation*}
Then, we can design the contrast function $\Psi_{n,\vv}(\cdot)$ as
\begin{equation*}
\begin{split}
\Psi_{n,\vv}(\theta)&=\vv^{-2}\dd^{-1}|\si^{-1}(\hat Y^\vv_{t_{k-1}},\mathscr{L}_{\hat Y^\vv_{t_{k-1}}})P_k(\theta)|^2\\
&=\vv^{-2}\dd^{-1}P_k^*(\theta)((\si^{-1})^*\si^{-1})(\hat Y^\vv_{t_{k-1}},\mathscr{L}_{\hat Y^\vv_{t_{k-1}}})P_k(\theta)\\
&=\vv^{-2}\dd^{-1}P_k^*(\theta)\LL_{k-1}^{-1}P_k(\theta).
\end{split}
\end{equation*}
Motivated by the invertible setup above, we establish the contrast
function for the setting that the diffusion $\si(\cdot,\cdot)$ need
not to be invertible; see \eqref{eq2} for further details. On the
other hand, if $b(\cdot,\cdot,\theta)$ is explicit w.r.t. the
parameter $\theta$, then the LSE $\hat\theta_{n,\vv}$ can indeed be
obtained by Fermat's theorem. }
\end{rem}

\begin{rem}
{\rm   Formally, the contrast function $\Psi_{n,\vv}$ can be defined
  as in \eqref{eq2} with $\hat Y^\vv_{t_k}$ being replaced by
$X_{t_k}^\vv$ in \eqref{eq3}. Nevertheless,  $X_{t_k}^\vv$ cannot be
available provided that $(X^\vv(t))_{t\in[0,T]}$ is observed only at
the points  $t=k\dd$. So, in our paper, we approximate the window
process $X_{t_k}^\vv$  via the linear interpolation; see \eqref{w2}
for more details. }
\end{rem}

\begin{rem}
{\rm We remark that our contrast function is established on the
basis of EM scheme. With regard to path-distribution dependent SDEs,
if the global Lipschitz condition ({\bf A1}) is replaced by the
monotone condition, then the contrast function \eqref{eq2} will no
longer work  due to the fact that the EM numerical solution will
explode in finite time. So, for such case, we need to establish the
LSE for the unknown parameter based on the new contrast function,
which will be reported in our forthcoming paper. }
\end{rem}

\section{The consistency of LSE}\label{sec2}
 First of all, let's consider the following deterministic
ordinary differential equation
\begin{equation}\label{eq6}
\d X(t)=b(X_t^0,\mathscr{L}_{X_t^0},\theta_0)\d
t,~~~t>0,~~~X_0^0=\xi\in\C.
\end{equation}
Under ({\bf A1}), \eqref{eq6} possesses  a unique solution
$(X^0(t))_{ t\ge-r_0}$. Herein, it is worth pointing out that
\eqref{eq1} and \eqref{eq6} share the same initial datum. For the
sake of notation brevity, for a random variable $\zeta\in\C$ with
$\mathscr{L}_\zeta\in\mathcal {P}_2(\C)$, let
\begin{equation}\label{c1}
\Lambda(\zeta,\theta,\theta_0) =b(\zeta,\mathscr{L}_\zeta,\theta_0)
-b(\zeta,\mathscr{L}_\zeta,\theta)~~~\mbox{ and
}~~~\hat\si(\zeta)=(\si\si^*)^{-1}(\zeta,\mathscr{L}_\zeta).
\end{equation}
Set, for any $\theta\in\Theta,$
\begin{equation}\label{e3}
\Xi(\theta):=\int_0^T\Lambda^*(X_t^0,\theta,\theta_0)\hat\si(X_t^0)\Lambda(X_t^0,\theta,\theta_0)\d
t,
\end{equation}
where $(X_t)$ is the segment process generated by the solution
$(X(t))$ to \eqref{eq6}.

Our first main result,   which is concerned with the consistency of
the LSE of $\theta\in\Theta$, is stated as below.
\begin{thm}\label{th1}
{\rm Let ({\bf A1})-({\bf A3}) hold  and assume further
$\Xi(\theta)>0$ for any $\theta\in\bar\Theta$. Then
\begin{equation*}
\hat\theta_{n,\vv}\rightarrow\theta_0~~~~\mbox{ in probability as }
 \vv\rightarrow0 ~~\mbox{ and }~~n\to\8.
\end{equation*}
}
\end{thm}

The proof of Theorem \ref{th1} is based on several auxiliary lemmas
below.


\begin{lem}\label{lem0}
{\rm Under ({\bf A1}), for any $p>0$, there exists $C_{p,T}>0$ such
that

\begin{equation}\label{r4}
\sup_{0\le t\le T}\E\|Y^\vv_{\lfloor t/\dd\rfloor\dd}\|_\8^p \le
C_{p,T}(1+\|\xi\|_\8^p),
\end{equation}
and
\begin{equation}\label{0r4}
\E\Big(\sup_{-r_0\le t\le T}|X^\vv(t)|^p\Big)\le
C_{p,T}(1+\|\xi\|_\8^p).
\end{equation}

}
\end{lem}
\begin{proof}
By H\"older's inequality, it is sufficient to show that \eqref{r4}
and \eqref{0r4} holds, respectively, for  any $p\ge2.$ Herein, we
only focus on the argument of \eqref{r4}   since \eqref{0r4} can be
done in a similar way.

Define the continuous-time EM scheme associated with \eqref{eq1}
\begin{equation}\label{r1}
\d\tt Y^\vv(t)=b(\hat Y^\vv_{\lfloor
t/\dd\rfloor\dd},\mathscr{L}_{\hat Y^\vv_{\lfloor
t/\dd\rfloor\dd}},\theta)\d t+\vv\si(\hat Y^\vv_{\lfloor
t/\dd\rfloor\dd},\mathscr{L}_{\hat Y^\vv_{\lfloor
t/\dd\rfloor\dd}})\d B(t),~~~t>0
\end{equation}
with $\tt Y^\vv(t)=X^\vv(t)=\xi(t)$ for $ t\in[-r_0,0],$ where $\hat
Y^\vv_{\lfloor t/\dd\rfloor\dd}(\cdot)$ is defined as in \eqref{w2}.
It is straightforward to check $\tt Y^\vv(k\dd)=Y^\vv(k\dd)$ for any
integer $k\in[-M,n]$. For any $t\in[0,T]$, a direct calculation
shows from \eqref{w2} that
\begin{equation}\label{r5}
\begin{split}
&\|\hat Y^\vv_{\lfloor t/\dd\rfloor\dd}\|_\8\\&=\sup_{-r_0\le
v\le0}|\hat Y^\vv_{\lfloor t/\dd\rfloor\dd}(v)|\\
&=\max_{k=0,\cdots,M-1}\sup_{-(k+1)\dd\le
v\le-k\dd}\Big|\ff{(k+1)\dd+v}{\dd}Y^\vv((\lfloor
t/\dd\rfloor-k)\dd)-\ff{k\dd+v}{\dd}Y^\vv((\lfloor
t/\dd\rfloor-k-1)\dd)\Big|\\
&\le\max_{k=0,\cdots,M-1}\sup_{-(k+1)\dd\le v\le-k\dd}\Big(|\tt
Y^\vv(\lfloor t/\dd\rfloor\dd-k\dd)|+|\tt Y^\vv(\lfloor
t/\dd\rfloor\dd-(k+1)\dd)|\Big)\\
&\le2\sup_{-r_0\le s\le t}|\tt Y^\vv(s)|,
\end{split}
\end{equation}
where in the first inequality  we have used the facts that  $\tt
Y^\vv(k\dd)=Y^\vv(k\dd)$ for any integer $k\in[-M,n]$ and that, for
any $v\in[-(k+1)\dd,-k\dd]$,
\begin{equation*}
\ff{(k+1)\dd+v}{\dd}\in[0,1]~~~\mbox{ and }~~~
\ff{k\dd+v}{\dd}\in[-1,0].
\end{equation*}
From ({\bf A1}), one has, for any $\zeta\in\C$ and $\mu\in\mathcal
{P}_2(\C)$,
\begin{equation}\label{r2}
|b(\zeta,\mu,\theta)|^2\le2\,\Big\{\aa_1\|\zeta\|_\8^2+\aa_2\mathbb{W}_2(\mu,\dd_{\zeta_0})^2+|b(\zeta_0,\dd_{\zeta_0},\theta)|^2\Big\},
\end{equation}
and
\begin{equation}\label{r3}
\|\si(\zeta,\mu)\|^2\le2\,\Big\{\bb_1\|\zeta\|_\8^2+\bb_2\mathbb{W}_2(\mu,\dd_{\zeta_0})^2+\|\si(\zeta_0,\dd_{\zeta_0})\|^2\Big\},
\end{equation}
where $\zeta_0(s)={\bf0}\in\R^d$ for any $s\in[-r_0,0].$ For any
$p\ge2,$ by H\"older's inequality and Burkhold-Davis-Gundy's (BDG's
for brevity) inequality (see, e.g., \cite[Theorem 7.3, P.40]{M08}),
we deduce from  \eqref{r2} and \eqref{r3} that
\begin{equation*}
\begin{split}
\Gamma(t):&=1+\E\Big(\sup_{-r_0\le s\le t}|\tt Y^\vv(s)|^p\Big)\\
&\le1+ c\,\|\xi\|_\8^p+c\,t^{p-1}\int_0^t\E|b(\hat Y^\vv_{\lfloor
s/\dd\rfloor\dd},\mathscr{L}_{\hat Y^\vv_{\lfloor
s/\dd\rfloor\dd}},\theta)|^p\d s+c\,\E\Big(\int_0^t\|\si(\hat
Y^\vv_{\lfloor s/\dd\rfloor\dd},\mathscr{L}_{\hat Y^\vv_{\lfloor
s/\dd\rfloor\dd}})\|^2\d s\Big)^{p/2}\\
&\le1+ c\,\|\xi\|_\8^p+c(t^{p-1}+t^{\ff{p-2}{2}})\int_0^t\{\E|b(\hat
Y^\vv_{\lfloor s/\dd\rfloor\dd},\mathscr{L}_{\hat Y^\vv_{\lfloor
s/\dd\rfloor\dd}},\theta)|^p+\E\|\si(\hat Y^\vv_{\lfloor
s/\dd\rfloor\dd},\mathscr{L}_{\hat Y^\vv_{\lfloor
s/\dd\rfloor\dd}})\|^p\}\d s\\
&\le1+
c\,\|\xi\|_\8^p+c(t^{p-1}+t^{\ff{p-2}{2}})\int_0^t\{1+\E\|\hat
Y^\vv_{\lfloor
s/\dd\rfloor\dd}\|_\8^p+\mathbb{W}_2(\mathscr{L}_{\hat
Y^\vv_{\lfloor
s/\dd\rfloor\dd}},\dd_{\zeta_0})^p\}\d s\\
&\le1+
c\,\|\xi\|_\8^p+c(t^{p-1}+t^{\ff{p-2}{2}})\int_0^t\{1+\E\|\hat
Y^\vv_{\lfloor s/\dd\rfloor\dd}\|_\8^p\}\d s.
\end{split}
\end{equation*}
This, together with \eqref{r5}, leads to
\begin{equation*}
\Gamma(t) \le1+
c\,\|\xi\|_\8^p+c(t^{p-1}+t^{\ff{p-2}{2}})\int_0^t\Gamma(s)\d s.
\end{equation*}
Then, the desired assertion \eqref{r4} follows from Gronwall's
inequality and \eqref{r5}.
\end{proof}

\begin{lem}\label{le5}
{\rm Let ({\bf A1})  be satisfied. Then, there is a constant $C_T>0$
such that
\begin{equation}\label{r6}
\sup_{0\le t\le T}\E\|X_t^\vv-X_t^0\|_\8^2\le C_T\,\vv^2.
\end{equation}
}
\end{lem}
\begin{proof}
 Note that
\begin{equation*}
\begin{split}
\E\|X_t^\vv-X_t^0\|_\8^2 \le  \E\Big(\sup_{0\le s\le
t}|X^\vv(s)-X^0(s)|^2\Big)=:A(t,\vv),
\end{split}
\end{equation*}
where we have used  $X^\vv_0=X_0^0=\xi$.
 By H\"older's inequality, Doob's submartingale inequality as well
 as It\^o's isometry, we obtain from ({\bf A1}) and \eqref{r3} that
\begin{equation*}
\begin{split}
A(t,\vv)&\le2\,t\int_0^t\E|b(X_s^\vv,\mathscr{L}_{X_s^\vv},\theta_0)-b(X_s^0,\mathscr{L}_{X_s^0},\theta_0)|^2\d
s+2\,\vv^2\,\E\Big(\sup_{0\le s\le t}\Big|\int_0^s
\si(X_u^\vv,\mathscr{L}_{X_u^\vv})\d B(u)\Big|^2\Big)\\
&\le2\,t\int_0^t\E|b(X_s^\vv,\mathscr{L}_{X_s^\vv},\theta_0)-b(X_s^0,\mathscr{L}_{X_s^0},\theta_0)|^2\d
s+8\,\vv^2 \int_0^t \E\|\si(X_s^\vv,\mathscr{L}_{X_s^\vv})\|^2\d
s\\
&\le2\,t\int_0^t\{\aa_1\E\|X_s^\vv-X_s^0\|^2_\8+\aa_2\mathbb{W}_2(\mathscr{L}_{X_s^\vv},\mathscr{L}_{X_s^0})^2\}\d
s\\
&\quad+c\,\vv^2
\int_0^t\{1+\E\|X_s^\vv\|_\8^2+\mathbb{W}_2(\mathscr{L}_{X_s^\vv},\dd_{\zeta_0})^2\}
\d s\\
&\le c\,t\int_0^t\E\|X_s^\vv-X_s^0\|^2_\8\d s+c\,\vv^2
\int_0^t\{1+\E\|X_s^\vv\|_\8^2\}
\d s\\
&\le c\,t\int_0^tA(s,\vv)\d s+c(1+C_{2,T})\,\vv^2t,
\end{split}
\end{equation*}
where we have used \eqref{0r4} in the last display. As a result,
 \eqref{r6} holds true by Gronwall's inequality.
\end{proof}

\begin{lem}
{\rm Assume that ({\bf A1}) and ({\bf A3}) hold. Then, for any
$\bb\in(0,1)$, there exists $c_\bb>0$ such that
\begin{equation}\label{y2}
\sup_{0\le t\le T}\E\|\hat Y^\vv_{\lfloor
t/\dd\rfloor\dd}-X_t^0\|_\8^2\le c\,\vv^2+c_\bb\dd^\bb.
\end{equation}
}
\end{lem}

\begin{proof}
Due to \eqref{r6}, for any $t\in[0,T],$
\begin{equation}\label{r7}
\begin{split}
\E\|\hat Y^\vv_{\lfloor t/\dd\rfloor\dd}-X_t^0\|_\8^2&\le3\{\E\|\hat
Y^\vv_{\lfloor t/\dd\rfloor\dd}-\tt Y_t^\vv\|_\8^2+\E\|\tt
Y_t^\vv-X_t^\vv\|_\8^2+\E\|X_t^\vv-X_t^0\|_\8^2\}\\
&\le c\{\vv^2+\E\|\hat Y^\vv_{\lfloor t/\dd\rfloor\dd}-\tt
Y_t^\vv\|_\8^2+\E\|\tt Y_t^\vv-X_t^\vv\|_\8^2\}.
\end{split}
\end{equation}
Next,  exploiting H\"older's inequality, Doob's submartingale
inequality and It\^o's isometry, we derive from ({\bf A1}) and
$X^\vv_0=\tt Y_0^\vv=\xi$ that
\begin{equation*}
\begin{split}
\E\|X_t^\vv-\tt Y_t^\vv\|_\8^2&\le\E\Big(\sup_{0\le s\le
t}|X^\vv(s)-\tt
Y^\vv(s)|^2\Big)\\
&\le2\,t\int_0^t\E|b(X_s^\vv,\mathscr{L}_{X_s^\vv},\theta)-b(\hat
Y^\vv_{\lfloor s/\dd\rfloor\dd},\mathscr{L}_{\hat Y^\vv_{\lfloor
s/\dd\rfloor\dd}},\theta)|^2\d s\\
&\quad+8\,\vv^2\int_0^t\E\|\si(X_s^\vv,\mathscr{L}_{X_s^\vv})-\si(\hat
Y^\vv_{\lfloor s/\dd\rfloor\dd},\mathscr{L}_{\hat Y^\vv_{\lfloor
s/\dd\rfloor\dd}})\|^2\d s\\
&\le2\,t\int_0^t\{\aa_1\E\|X_s^\vv-\hat Y^\vv_{\lfloor
s/\dd\rfloor\dd}\|_\8^2+\aa_2\mathbb{W}_2(\mathscr{L}_{X_s^\vv},\mathscr{L}_{\hat
Y^\vv_{\lfloor
s/\dd\rfloor\dd}})^2\}\d s\\
&\quad+8\,\vv^2\int_0^t\{\bb_1\E\|X_s^\vv-\hat Y^\vv_{\lfloor
s/\dd\rfloor\dd}\|_\8^2+\bb_2\mathbb{W}_2(\mathscr{L}_{X_s^\vv},\mathscr{L}_{\hat
Y^\vv_{\lfloor s/\dd\rfloor\dd}})^2\}\d s.
\end{split}
\end{equation*}
Consequently, we obtain from $\vv\in(0,1)$ that
\begin{equation*}
\begin{split}
\E\|X_t^\vv-\tt Y_t^\vv\|_\8^2 &\le c\,(1+t)\int_0^t\E\|X_s^\vv-\hat
Y^\vv_{\lfloor s/\dd\rfloor\dd}\|_\8^2\d
s\\
&\le c\,(1+t)\int_0^t\E\|X_s^\vv-\tt Y^\vv_s\|_\8^2\d
s+c\,(1+t)\int_0^t\E\|\tt Y^\vv_s-\hat Y^\vv_{\lfloor
s/\dd\rfloor\dd}\|_\8^2\d s.
\end{split}
\end{equation*}
Thus, Gronwall's inequality yields that
\begin{equation}\label{r8}
\E\|X_t^\vv-\tt Y_t^\vv\|_\8^2 \le c\sup_{0\le t\le T}\E\|\tt
Y^\vv_t-\hat Y^\vv_{\lfloor t/\dd\rfloor\dd}\|_\8^2.
\end{equation}
 Substituting \eqref{r8} into \eqref{r7} gives
that
\begin{equation}\label{r9}
\E\|\hat Y^\vv_{\lfloor t/\dd\rfloor\dd}-X_t^0\|_\8^2 \le
c\Big\{\vv^2+\sup_{t\in[0,T]}\E\|\tt Y_t^\vv-\hat Y^\vv_{\lfloor
t/\dd\rfloor\dd}\|_\8^2\Big\}.
\end{equation}
So, to achieve  \eqref{y2}, it remains  to show that, for any
$\bb\in(0,1)$, there exists $c_\bb>0$ such that
\begin{equation}\label{y4}
\sup_{t\in[0,T]}\E\|\tt Y_t^\vv-\hat Y^\vv_{\lfloor
t/\dd\rfloor\dd}\|_\8^2\le c_\bb\dd^\bb.
\end{equation}
For any $t\in[0,T)$, there exists an integer $k_0\in[0,n-1]$ such
that $t\in[k_0\dd,(k_0+1)\dd)$ so that  $\lfloor t/\dd\rfloor=k_0.$
By H\"older's inequality, for any $\bb\in(0,1),$
\begin{equation*}
\begin{split}
\E\|\tt Y_t^\vv-\hat Y^\vv_{k_0\dd}\|_\8^2&=\E\Big(\sup_{-r_0\le
v\le
0}|\tt Y^\vv(t+v)-\hat Y^\vv_{k_0\dd}(v)|^2\Big)\\
&\le\Big(\E\Big(\sup_{-r_0\le v\le
0}|\tt Y^\vv(t+v)-\hat Y^\vv_{k_0\dd}(v)|^{\ff{2}{1-\bb}}\Big)\Big)^{1-\bb}\\
&\le M^{1-\bb}\max_{k=0,\cdots,M-1}\,\Big(\E\Big(\sup_{-(k+1)\dd\le
v\le-k\dd}|\tt Y^\vv(t+v)-\hat
Y^\vv_{k_0\dd}(v)|^{\ff{2}{1-\bb}}\Big)\Big)^{1-\bb},
\end{split}
\end{equation*}
where $M>0$ is an integer such that $r_0=M\dd.$ For any
$v\in[-(k+1)\dd,-k\dd]$ with $k=0,\cdots,M-1$, it follows from
\eqref{w2} that
\begin{equation*}
\begin{split}
\tt Y^\vv(t+v)-\hat Y_{k_0\dd}^\vv(v)&=\ff{(k+1)\dd+v}{\dd}\Big(\tt
Y^\vv(t+v)-Y^\vv((k_0-k)\dd)\Big)\\
&\quad-\ff{k\dd+v}{\dd}\Big(\tt Y^\vv(t+v)-Y^\vv((k_0-k-1)\dd)\Big).
\end{split}
\end{equation*}
As a consequence, we deduce that
\begin{equation}\label{y5}
\begin{split}
&\E\|\tt Y_t^\vv-\hat Y^\vv_{k_0\dd}\|_\8^2\\ &\le c\,
M^{1-\bb}\max_{k=0,\cdots,M-1}\,\Big(\E\Big(\sup_{(k_0-k-1)\dd\le
s\le(k_0-k+1)\dd}|\tt
Y^\vv(s)-Y^\vv((k_0-k)\dd)|^{\ff{2}{1-\bb}}\Big)\Big)^{1-\bb}\\
&\quad+c\,M^{1-\bb}\max_{k=0,\cdots,M-1}\,\Big(\E\Big(\sup_{(k_0-k-1)\dd\le
s\le(k_0-k+1)\dd}|\tt
Y^\vv(s)-Y^\vv((k_0-k-1)\dd)|^{\ff{2}{1-\bb}}\Big)\Big)^{1-\bb}\\
&=:A_1(\vv,\dd)+A_2(\vv,\dd).
\end{split}
\end{equation}

 For any $t\in[l\dd,(l+1)\dd]$ with $l=0,1,\cdots,n-1,$ we have
\begin{equation}\label{y3}
\begin{split}
\E\Big(\sup_{l\dd \le s \le t}|\tt Y^\vv(s)-\tt
Y^\vv(l\dd)|^{\ff{2}{1-\bb}}\Big)&\le c\,\Big\{
\dd^{\ff{2}{1-\bb}}\E|b(\hat Y^\vv_{l\dd},\mathscr{L}_{\hat  Y^\vv_{l\dd}},\theta)|^{\ff{2}{1-\bb}}\\
&\quad+ \E\|\si(\hat
 Y^\vv_{l\dd},\mathscr{L}_{\hat  Y^\vv_{l\dd}})\|^{\ff{2}{1-\bb}}\E\Big(\sup_{l\dd
\le s \le t}|B(s)-B(l\dd)|^{\ff{2}{1-\bb}}\Big)\Big\}\\
&=c\,\Big\{
\dd^{\ff{2}{1-\bb}}\E|b(\hat  Y^\vv_{l\dd},\mathscr{L}_{\hat  Y^\vv_{l\dd}},\theta)|^{\ff{2}{1-\bb}}\\
&\quad+ \E\|\si(\hat  Y^\vv_{l\dd},\mathscr{L}_{\hat
 Y^\vv_{l\dd}})\|^{\ff{2}{1-\bb}}\E\Big(\sup_{0
\le s \le t-l\dd}|B(s)|^{\ff{2}{1-\bb}}\Big)\Big\}\\
\end{split}
\end{equation}
where we have used the fact that $\hat Y^\vv_{l\dd}$ is independent
of $B(t)-B(l\dd)$ for any $t\in [l\dd,(l+1)\dd]$ in the first
inequality and the independent increment property of Brownian motion
in the last display.

Let $(e_i)_{1\le i\le m}$ be the standard orthogonal basis of
$\R^m$. Note that $B_i(t):=\<B(t),e_i\>$ is a scalar Brownian motion
and
\begin{equation*}
\P\Big(\sup_{0\le s\le t}B_i(s)\ge x\Big)=2\P(B_i(t)\ge x),
\end{equation*}
see, for instance, \cite[Theorem 3.15]{K12}. Whence, for any $p>1$,
we deduce that
\begin{equation*}
\begin{split}
\E\Big(\sup_{0\le s\le t}|B(s)|^p\Big)&=\int_0^\8\P\Big(\sup_{0\le
s\le t}|B(s)|^p\ge x\Big)\d x\\
&=2\sum_{i=1}^m\int_0^\8\P\Big(B_i(t)\ge
x^{1/p}/m^{1/2}\Big)\d x\\
&=\ff{2}{\ss{2\pi t}}\sum_{i=1}^m\int_0^\8\d
x\int_{\ff{x^{\ff{1}{p}}}{m^{\ff{1}{2}}}}^\8\e^{-\ff{y^2}{2t}}\d
y\\
&\le\ff{2m^{1/2}}{\ss{2\pi t}}\sum_{i=1}^m\int_0^\8x^{-\ff{1}{p}}\d
x\int_{\ff{x^{\ff{1}{p}}}{m^{\ff{1}{2}}}}^\8y\e^{-\ff{y^2}{2t}}\d
y\\
&=\ff{2m^{1/2}t}{\ss{2\pi
t}}\sum_{i=1}^m\int_0^\8x^{-\ff{1}{p}}\e^{-\ff{x^{2/p}}{2mt}}\d x\\
&\le c\, t^{\ff{p}{2}},
\end{split}
\end{equation*}
where in the last step we have utilized the Gamma function
\begin{equation*}
\Gamma(\aa)=\int_0^\8\e^{-x}x^{\aa-1}\d x,~~~~\aa>0.
\end{equation*}
This, combining \eqref{r2} with  \eqref{r3} and \eqref{y3}, yields
that, for any $t\in[l\dd,(l+1)\dd]$,
\begin{equation}\label{y1}
\begin{split}
\E\Big(\sup_{l\dd \le s \le t}|\tt Y^\vv(s)-\tt
Y^\vv(l\dd)|^{\ff{2}{1-\bb}}\Big)&\le
c\dd^{\ff{1}{1-\bb}}\{\E|b(\hat Y^\vv_{l\dd},\mathscr{L}_{\hat
Y^\vv_{l\dd}},\theta)|^{\ff{2}{1-\bb}}+
\E\|\si(\hat Y^\vv_{l\dd},\mathscr{L}_{\hat Y^\vv_{l\dd}})\|^{\ff{2}{1-\bb}}\}\\
&\le c\dd^{\ff{1}{1-\bb}}\{1+\E\|\hat Y^\vv_{l\dd}\|_\8^{\ff{2}{1-\bb}}+\mathbb{W}_2(\mathscr{L}_{\hat Y^\vv_{l\dd}},\dd_{\zeta_0})^{\ff{2}{1-\bb}}\}\\
&\le c\dd^{\ff{1}{1-\bb}}\{1+\E\|\hat Y^\vv_{l\dd}\|_\8^{\ff{2}{1-\bb}}\}\\
&\le c\dd^{\ff{1}{1-\bb}},
\end{split}
\end{equation}
where in the last procedure we have exploited \eqref{r4}.

In the sequel, we divide three cases to show the estimates on
$A_1(\vv,\dd)$ and $A_2(\vv,\dd)$.

\noindent{{\bf Case 1: $k\ge k_0+1$}.} With regard to such case,
$(k_0-k+1)\dd\in[-r_0,0].$ We infer from ({\bf A1}) and \eqref{y5},
in addition to $M\dd=r_0 $, that
\begin{equation*}
A_1(\vv,\dd)+A_2(\vv,\dd)\le cM^{1-\bb}\dd=c\,r_0^{1-\bb}\dd^\bb.
\end{equation*}

\noindent{{\bf Case 2: $k_0=k$}.} For this case,
$t\in[k\dd,(k+1)\dd)$. Again, one gets from \eqref{y5} that
\begin{equation*}
\begin{split}
&A_1(\vv,\dd)+A_2(\vv,\dd)\\
&\le c\, M^{1-\bb}\max_{k=0,\cdots,M-1}\,\Big(\E\Big(\sup_{-\dd\le
s\le\dd}|\tt
Y^\vv(s)-\tt Y^\vv(0)|^{\ff{2}{1-\bb}}\Big)\Big)^{1-\bb}\\
&\quad+c\,M^{1-\bb}\max_{k=0,\cdots,M-1}\,\Big(\E\Big(\sup_{-\dd\le
s\le\dd}|\tt Y^\vv(s)-\tt
Y^\vv(-\dd)|^{\ff{2}{1-\bb}}\Big)\Big)^{1-\bb},
\end{split}
\end{equation*}
where we have employed $Y^\vv(t)=\tt Y(t), t\in[-r_0,0].$ This,
besides ({\bf A3}) and \eqref{y1},   implies that
\begin{equation*}
\begin{split}
A_1(\vv,\dd)+A_2(\vv,\dd)&\le c\dd^\bb+ c\,
M^{1-\bb}\max_{k=0,\cdots,M-1}\,\Big(\E\Big(\sup_{-\dd\le
s\le\dd}|\tt
Y^\vv(s)-\tt Y^\vv(0)|^{\ff{2}{1-\bb}}\Big)\Big)^{1-\bb}\\
&\le c\dd^\bb+ c\,
M^{1-\bb}\max_{k=0,\cdots,M-1}\,\Big(\E\Big(\sup_{0\le s \le
\dd}|\tt Y^\vv(s)-\tt
Y^\vv(0)|^{\ff{2}{1-\bb}}\Big)\Big)^{1-\bb}\\
&\le c\dd^\bb+ c\, M^{1-\bb}\dd\\
&\le c\dd^\bb.
\end{split}
\end{equation*}
\noindent{{\bf Case 3: $k\le k_0-1$}.} Also, by making use of
\eqref{y1}, it follows that
\begin{equation*}
\begin{split}
&A_1(\vv,\dd)+A_2(\vv,\dd)\\
&\le c\,
M^{1-\bb}\max_{k=0,\cdots,M-1}\,\Big(\E\Big(\sup_{(k_0-k-1)\dd\le
s\le (k_0-k+1)\dd}|\tt
Y^\vv(s)-\tt Y^\vv((k_0-k-1)\dd)|^{\ff{2}{1-\bb}}\Big)\Big)^{1-\bb}\\
&\quad+c\, M^{1-\bb}\max_{k=0,\cdots,M-1}\,\Big(\E|\tt
Y^\vv((k_0-k-1)\dd)-\tt Y^\vv((k_0-k)\dd)|^{\ff{2}{1-\bb}}\Big)^{1-\bb}\\
&\le c\,
M^{1-\bb}\max_{k=0,\cdots,M-1}\,\Big(\E\Big(\sup_{(k_0-k-1)\dd\le
s\le (k_0-k)\dd}|\tt
Y^\vv(s)-\tt Y^\vv((k_0-k-1)\dd)|^{\ff{2}{1-\bb}}\Big)\Big)^{1-\bb}\\
 &\quad+c\,
M^{1-\bb}\max_{k=0,\cdots,M-1}\,\Big(\E\Big(\sup_{(k_0-k)\dd\le s\le
(k_0-k+1)\dd}|\tt
Y^\vv(s)-\tt Y^\vv((k_0-k)\dd)|^{\ff{2}{1-\bb}}\Big)\Big)^{1-\bb}\\
&\quad+c\, M^{1-\bb}\max_{k=0,\cdots,M-1}\,\Big(\E|\tt
Y^\vv((k_0-k-1)\dd)-\tt Y^\vv((k_0-k)\dd)|^{\ff{2}{1-\bb}}\Big)^{1-\bb}\\
&\le c\dd^\bb.
\end{split}
\end{equation*}
By summing up the three cases above, \eqref{y4} holds true.
\end{proof}

\begin{lem}\label{lem2}
{\rm Let ({\bf A1})-({\bf A3}) hold. Then,
\begin{equation}\label{a7}
\begin{split}
&\dd\sum_{k=1}^n\Lambda^*(\hat Y_{t_{k-1}}^\vv,\theta,\theta_0)\hat\si(\hat Y_{t_{k-1}}^\vv)\Lambda(\hat Y_{t_{k-1}}^\vv,\theta,\theta_0)\\
&\rightarrow\Xi(\theta):=\int_0^T\Lambda^*(X_s^0,\theta,\theta_0)\hat\si(X_s^0)\Lambda(X_s^0,\theta,\theta_0)\d
s
\end{split}
\end{equation}
in $L^1$ as $\vv\rightarrow0$ and $\delta\rightarrow0 $ (i.e.,
$n\rightarrow\8$), in which $\Lambda(\cdot)$ and $\hat\si(\cdot)$
are introduced in \eqref{c1}.

}
\end{lem}

\begin{proof}
It is straightforward to see that
\begin{equation*}
\begin{split}
&\dd\sum_{k=1}^n\Lambda^*(\hat
Y_{t_{k-1}}^\vv,\theta,\theta_0)\hat\si(\hat
Y_{t_{k-1}}^\vv)\Lambda(\hat Y_{t_{k-1}}^\vv,\theta,\theta_0)-
\int_0^T\Lambda^*(X_s^0,\theta,\theta_0)\hat\si(X_s^0)\Lambda(X_s^0,\theta,\theta_0)\d
s\\
&=\int_0^T\Big\{\Lambda^*(\hat Y_{\lfloor
s/\dd\rfloor\dd}^\vv,\theta,\theta_0)\hat\si(\hat Y_{\lfloor
s/\dd\rfloor\dd}^\vv)\Lambda(\hat Y_{\lfloor
s/\dd\rfloor\dd}^\vv,\theta,\theta_0)
-\Lambda^*(X_s^0,\theta,\theta_0)\hat\si(X_s^0)\Lambda(X_s^0,\theta,\theta_0)\Big\}\d
s\\
&=\int_0^T \Big(\Lambda(\hat Y_{\lfloor
s/\dd\rfloor\dd}^\vv,\theta,\theta_0)-\Lambda(X_s^0,\theta,\theta_0)\Big)^*\hat\si(\hat
Y_{\lfloor s/\dd\rfloor\dd}^\vv)\Lambda(\hat Y_{\lfloor
s/\dd\rfloor\dd}^\vv,\theta,\theta_0)\d s\\
&\quad+\int_0^T\Lambda^*(X_s^0,\theta,\theta_0)
 \Big(\hat\si(\hat Y_{\lfloor
s/\dd\rfloor\dd}^\vv)-\hat\si(X_s^0)\Big)\Lambda(\hat Y_{\lfloor
s/\dd\rfloor\dd}^\vv,\theta,\theta_0) \d
s\\
&\quad+\int_0^T\Lambda^*(X_s^0,\theta,\theta_0)
\hat\si(X_s^0)\Big(\Lambda(\hat Y_{\lfloor
s/\dd\rfloor\dd}^\vv,\theta,\theta_0)-\Lambda(X_s^0,\theta,\theta_0)\Big)
\d
s\\
&=:J_1(\vv,\dd)+J_2(\vv,\dd)+J_3(\vv,\dd).
\end{split}
\end{equation*}
Next, for any random variables $\zeta_1,\zeta_2\in\C$ with
$\mathscr{L}_{\zeta_1},\mathscr{L}_{\zeta_2}\in\mathcal {P}_2(\C)$,
observe from ({\bf A1}) that
\begin{equation}\label{a1}
\begin{split}
|\Lambda(\zeta_1,\theta,\theta_0)-\Lambda(\zeta_2,\theta,\theta_0)|&\le|b(\zeta_1,\mathscr{L}_{\zeta_1},\theta_0)-b(\zeta_2,\mathscr{L}_{\zeta_2},\theta_0)|
+|b(\zeta_1,\mathscr{L}_{\zeta_1},\theta)-b(\zeta_2,\mathscr{L}_{\zeta_2},\theta)|\\
&\le
c\Big\{\|\zeta_1-\zeta_2\|_\8+\mathbb{W}_2(\mathscr{L}_{\zeta_1},\mathscr{L}_{\zeta_2})\Big\}.
\end{split}
\end{equation}
For a  random variable  $\zeta\in\C$ with
$\mathscr{L}_\zeta\in\mathcal {P}_2(\C)$,  employing ({\bf A2})
gives that
\begin{equation}\label{a6}
\begin{split}
\|\hat\si(\zeta)\|&\le\|\hat\si(\zeta)-\hat\si(\zeta_0)\|+\|\hat\si(\zeta_0)\|
\le c\,\Big\{1+
 \|\zeta\|_\8+ \mathbb{W}_2(\mathscr{L}_\zeta,\dd_{\zeta_0})\Big\}.
 \end{split}
\end{equation}
Consequently, combining \eqref{r2} with \eqref{a1} and \eqref{a6},
 we deduce that
\begin{align*}
&|J_1(\vv,\dd)|+|J_3(\vv,\dd)|\\&\le c\int_0^T\Big\{\|\hat
Y_{\lfloor
s/\dd\rfloor\dd}^\vv-X_s^0\|_\8+\mathbb{W}_2(\mathscr{L}_{\hat
Y_{\lfloor
s/\dd\rfloor\dd}^\vv},\mathscr{L}_{X_s^0})\Big\}\\
&\quad\times\Big\{1+\|X_s^0\|_\8+
 \|\hat Y_{\lfloor
s/\dd\rfloor\dd}^\vv\|_\8+ \mathbb{W}_2(\mathscr{L}_{\hat Y_{\lfloor
s/\dd\rfloor\dd}^\vv},\dd_{\zeta_0})\Big\}^2\d s\\
&\le c\int_0^T\Big\{\|\hat Y_{\lfloor
s/\dd\rfloor\dd}^\vv-X_s^0\|_\8+\ss{\E\|\hat Y_{\lfloor
s/\dd\rfloor\dd}^\vv-X_s^0\|_\8^2}\Big\}\\
&\quad\times\Big\{1+\|X_s^0\|_\8^2+
 \|\hat Y_{\lfloor
s/\dd\rfloor\dd}^\vv\|_\8^2 + \E\|\hat Y_{\lfloor
s/\dd\rfloor\dd}^\vv\|_\8^2\Big\}\d s.
\end{align*}
 This, together with \eqref{r4} and
\eqref{y2} as well as H\"older's inequality, implies that
\begin{equation}\label{a8}
\begin{split}
&\E|J_1(\vv,\dd)|+\E|J_3(\vv,\dd)|\\&\le c\,\int_0^T\ss{\E\|\hat
Y_{\lfloor
s/\dd\rfloor\dd}^\vv-X_s^0\|_\8^2}\Big\{1+\|X_s^0\|_\8^4+\E
 \|\hat Y_{\lfloor
s/\dd\rfloor\dd}^\vv\|_\8^4\Big\}\d s\\
&\rightarrow0
\end{split}
\end{equation}
as  $\vv\rightarrow0$ and $\dd\rightarrow0$. Next,  making use of
({\bf A2}) and \eqref{r2}, we derive that
\begin{equation*}
\begin{split}
 |J_2(\vv,\dd)|&\le c\int_0^T(1+
 \|X_s^0\|_\8)(1+
 \|\hat Y_{\lfloor
s/\dd\rfloor\dd}^\vv\|_\8 + \mathbb{W}_2(\mathscr{L}_{\hat
Y_{\lfloor
s/\dd\rfloor\dd}^\vv},\dd_{\zeta_0}))\\
&\quad\times\Big(\|\hat Y_{\lfloor
s/\dd\rfloor\dd}^\vv-X_s^0\|_\8+\ss{\E\|\hat Y_{\lfloor
s/\dd\rfloor\dd}^\vv-X_s^0\|_\8^2}\Big)\d s.
\end{split}
\end{equation*}
Again, using \eqref{r4} and \eqref{y2} and utilizing H\"older's
inequality gives that  \begin{equation}\label{a9}
\begin{split}
\E|J_2(\vv,\dd)|&\le c\,\int_0^T\ss{\E\|\hat Y_{\lfloor
s/\dd\rfloor\dd}^\vv-X_s^0\|_\8^2}\Big\{1+\E
 \|\hat Y_{\lfloor
s/\dd\rfloor\dd}^\vv\|_\8^2\Big\}\d s\\
&\rightarrow0
\end{split}
\end{equation}
whenever $\vv\rightarrow0$  and $\dd\rightarrow0$. Hence, \eqref{a7}
follows immediately from \eqref{a8} and \eqref{a9}.
\end{proof}

\begin{lem}\label{lem4}
{\rm Let ({\bf A1})-({\bf A3}) hold.  Then,
\begin{equation}\label{s5}
\sum_{k=1}^n\Lambda^*(\hat
Y_{t_{k-1}}^\vv,\theta,\theta_0)\hat\si(\hat
Y_{t_{k-1}}^\vv)P_k(\theta_0)\longrightarrow0
\end{equation}
in $L^1$ as $\vv\rightarrow0$, where $P_k$ is introduced in
\eqref{w1}. }
\end{lem}

\begin{proof}
Note that
\begin{equation*}
\begin{split}
\Upsilon(\vv,\dd):&=\sum_{k=1}^n\Lambda^*(\hat Y_{t_{k-1}}^\vv,\theta,\theta_0)\hat\si(Y_{t_{k-1}}^\vv)P_k(\theta_0)\\
&=\vv\sum_{k=1}^n\Lambda^*(\hat
Y_{t_{k-1}}^\vv,\theta,\theta_0)\hat\si(\hat Y_{t_{k-1}}^\vv)
\si(Y_{t_{k-1}}^\vv,\mathscr{L}_{\hat Y_{t_{k-1}}^\vv})(B(t_k)-B(t_{k-1}))\\
&=\vv\int_0^T\Lambda^*(\hat Y_{\lfloor
s/\dd\rfloor\dd}^\vv,\theta,\theta_0)\hat\si(\hat Y_{\lfloor
s/\dd\rfloor\dd}^\vv)\si(\hat Y_{\lfloor
s/\dd\rfloor\dd}^\vv,\mathscr{L}_{\hat Y_{\lfloor
s/\dd\rfloor\dd}^\vv})\d B(s).
\end{split}
\end{equation*}
Employing  H\"older's inequality and It\^o's isometry and taking
\eqref{r2}, \eqref{r3} and \eqref{a6} into account, we find that
\begin{equation}\label{s4}
\begin{split}
\E|\Upsilon(\vv,\dd)|&\le\vv\Big(\int_0^T\E\|\Lambda^*(\hat
Y_{\lfloor s/\dd\rfloor\dd}^\vv,\theta,\theta_0)\hat\si(\hat
Y_{\lfloor s/\dd\rfloor\dd}^\vv)\si(\hat Y_{\lfloor
s/\dd\rfloor\dd}^\vv,\mathscr{L}_{\hat Y_{\lfloor
s/\dd\rfloor\dd}^\vv})\|^2\d
s\Big)^{1/2}\\
&\le c\,\vv\Big(\int_0^T\Big\{1+\E
 \|\hat Y_{\lfloor
s/\dd\rfloor\dd}^\vv\|_\8^6 + \mathbb{W}_2(\mathscr{L}_{\hat
Y_{\lfloor s/\dd\rfloor\dd}^\vv},\dd_{\zeta_0})^6\Big\}\d
s\Big)^{1/2}\\
&\le c\,\vv\Big(\int_0^T\Big\{1+\E
 \|\hat Y_{\lfloor
s/\dd\rfloor\dd}^\vv\|_\8^6 \Big\}\d
s\Big)^{1/2}\\
&\le c \,\vv,
\end{split}
\end{equation}
where we have applied \eqref{r4} in  the last step. Therefore,
\eqref{s5} is now available from  \eqref{s4}.
\end{proof}

To make the content self-contained, we cite \cite[Theorem 5.9]{V98}
as the following lemma.

\begin{lem}\label{lem3}
{\rm   Let $(M_n)_{n\ge1}$ be random functions and   $M$   a fixed
function of $\theta$ such that, for any $\vv>0 $,
\begin{equation*}
\sup_{\theta\in\Theta}|M_n(\theta)-M(\theta)|\longrightarrow0~~~~\mbox{
in probability }
\end{equation*}
and
$
\sup_{|\theta-\theta_0|\ge\vv}M(\theta)<M(\theta_0).
$
Then, any sequence of estimators $\hat\theta_n$ with
$M_n(\hat\theta_n) \ge M_n(\theta_0)$ converges in probability to
$\theta_0$.
 }
\end{lem}

With Lemmas \ref{lem2}-\ref{lem3} in hand, we  are in the position
to complete the proof of Theorem \ref{th1}.
\begin{proof}[Proof of Theorem \ref{th1}]
From \eqref{eq2}, we infer that
\begin{equation}\label{s6}
\begin{split}
&\Phi_{n,\vv}(\theta)\\&=\delta^{-1}\sum_{k=1}^n\Big\{P_k^*(\theta)\hat\si(\hat Y_{t_{k-1}}^\vv)P_k(\theta)-P_k^*(\theta_0)\hat\si(\hat Y_{t_{k-1}}^\vv)P_k(\theta_0)\Big\}\\
&=\delta^{-1}\sum_{k=1}^n\Big\{\Big(P_k(\theta_0)+\Lambda(\hat
Y_{t_{k-1}}^\vv,\theta,\theta_0)\dd\Big)^*\hat\si(\hat
Y_{t_{k-1}}^\vv) \Big(P_k(\theta_0)
+\Lambda(\hat Y_{t_{k-1}}^\vv,\theta,\theta_0)\dd\Big)\\
&\quad -P_k^*(\theta_0)\hat\si(\hat Y_{t_{k-1}}^\vv)P_k(\theta_0)\Big\}\\
&=2\sum_{k=1}^n\Lambda^*(\hat
Y_{t_{k-1}}^\vv,\theta,\theta_0)\hat\si(\hat
Y_{t_{k-1}}^\vv)P_k(\theta_0)
+\dd\sum_{k=1}^n\Lambda^*(\hat Y_{t_{k-1}}^\vv,\theta,\theta_0)\hat\si(\hat Y_{t_{k-1}}^\vv)\Lambda(\hat Y_{t_{k-1}}^\vv,\theta,\theta_0)\\
&=:\Phi_{n,\vv}^{(1)}(\theta)+\Phi_{n,\vv}^{(2)}(\theta).
\end{split}
\end{equation}
In terms of Lemmas \ref{lem2} and \ref{lem4}, we deduce from
Chebyshev's inequality that
\begin{equation*}
\sup_{\theta\in\Theta}|-\Phi_{n,\vv}(\theta)-(-\Xi(\theta))|\rightarrow0~~~~\mbox{
in probability,}
\end{equation*}
where $\Xi(\cdot)$ is defined as in \eqref{a7}. On the other hand,
for any $\kk>0,$ notice that
\begin{equation*}
 \sup_{|\theta-\theta_0|\ge\kk}(-\Xi(\theta))<-\Xi(\theta_0)=0
\end{equation*}
due to $\Xi(\cdot)>0.$ Moreover, according to the notion of
$\hat\theta_{n,\vv}$, one has $-\Phi_{n,\vv}(
\hat\theta_{n,\vv})\ge-\Phi_{n,\vv}(\theta_0)=0.$ As far as our
present model is concerned, all of the assumptions in Lemma
\ref{lem3} with $M_n(\cdot)=-\Phi_{n,\vv}(\cdot)$ and
$M(\cdot)=-\Xi(\cdot)$ are fulfilled. As a consequence, we conclude
that $\hat\theta_{n,\vv}\rightarrow\theta_0$ in probability as
$\vv\rightarrow0$ and $n\rightarrow\8$, as required.
\end{proof}

\section{The asymptotic distribution of LSE}\label{sec3}
In this section, to begin, we recall some materials on derivatives
for matrix-valued functions and introduce some notation. For a
differentiable mapping $V=(V_1,\cdots,V_d)^*:\R^p\rightarrow\R^d$,
its gradient operator $(\nn_x V)(x)\in\R^d\otimes\R^p$ w.r.t. the
argument $x=(x_1,\cdots,x_p)^*\in\R^p$ is given by
\begin{equation}\label{A2}
(\nn_x V)(x)=\left(\begin{array}{cccc} \ff{\partial}{\partial
x_1}V_1(x)&\ff{\partial}{\partial
x_2}V_1(x)&\cdots&\ff{\partial}{\partial x_p}V_1(x)\\
 \ff{\partial}{\partial
x_1}V_2(x)&\ff{\partial}{\partial
x_2}V_2(x)&\cdots&\ff{\partial}{\partial x_p}V_2(x)\\
\cdots&\cdots&\cdots&\cdots\\
\ff{\partial}{\partial x_1}V_d(x)&\ff{\partial}{\partial
x_2}V_d(x)&\cdots&\ff{\partial}{\partial x_p}V_d(x)\\
 \end{array}
  \right).
\end{equation}
 If $V=(V_1,\cdots,V_d):\R^p\rightarrow (\R^d)^*$ (i.e., the
$d$-dimensional raw vector) is  differentiable, its gradient
operator $(\nn_x V)(x)\in\R^p\otimes\R^d$ w.r.t. the argument
$x=(x_1,\cdots,x_p)^*\in\R^p$ reads as follows
\begin{equation}\label{A3}
(\nn_x V)(x)=\left(\begin{array}{cccc} \ff{\partial}{\partial
x_1}V_1(x)&\ff{\partial}{\partial
x_1}V_2(x)&\cdots&\ff{\partial}{\partial x_1}V_d(x)\\
 \ff{\partial}{\partial
x_2}V_1(x)&\ff{\partial}{\partial
x_2}V_2(x)&\cdots&\ff{\partial}{\partial x_2}V_d(x)\\
\cdots&\cdots&\cdots&\cdots\\
\ff{\partial}{\partial x_p}V_1(x)&\ff{\partial}{\partial
x_p}V_2(x)&\cdots&\ff{\partial}{\partial x_p}V_d(x)\\
 \end{array}
  \right).
\end{equation}
So,   from \eqref{A2} and \eqref{A3}, one has $\nn_x V^*(x)=(\nn_x
V)^*(x)$ for a differentiable function $V:\R^p\rightarrow\R^d$. Let
$V=(V_{ij})_{p\times d}:\R\rightarrow\R^p\otimes\R^d$ be
differentiable. Then, the derivative $\frac{\partial}{\partial
x}V(x)\in\R^p\otimes\R^d$ of the matrix-valued mapping $V$ w.r.t.
the scalar argument $x\in\R$ enjoys the form
\begin{equation}\label{A1}
\ff{\partial}{\partial x}V(x)=\left(\begin{array}{cccc}
\ff{\partial}{\partial x}V_{11}(x)&\ff{\partial}{\partial
x}V_{12}(x)&\cdots&\ff{\partial}{\partial x}V_{1d}(x)\\
\ff{\partial}{\partial x}V_{21}(x)&\ff{\partial}{\partial
x}V_{22}(x)&\cdots&\ff{\partial}{\partial x}V_{2d}(x)\\
\cdots&\cdots&\cdots&\cdots\\
\ff{\partial}{\partial x}V_{p1}(x)&\ff{\partial}{\partial
x}V_{p2}(x)&\cdots&\ff{\partial}{\partial x}V_{pd}(x)\\
\end{array}\right).
\end{equation}
For a differentiable function $V=(V_{ij})_{p\times
d}:\R^p\rightarrow\R^p\otimes\R^d$, the gradient operator, denoted
by $\nn_x V(x)\in\R^p\otimes\R^{pd}$, of $V(x)$ w.r.t. the variable
$x=(x_1,\cdots,x_p)^*\in\R^p$ is formulated as
\begin{equation*}
(\nn_xV)(x)=\Big(\ff{\partial}{\partial
x_1}V(x),\ff{\partial}{\partial
x_2}V(x),\cdots,\ff{\partial}{\partial x_p}V(x)\Big),
\end{equation*}
where $\ff{\partial}{\partial x_i}V(x)$ is defined as in \eqref{A1}.
Moreover, for a differentiable function $V=(V_{ij})_{p\times
d}:\R^p\rightarrow \R^d$, we have
\begin{equation}\label{z1}
(\nn_x^{(2)}V^*)(x):=(\nn_x(\nn_x V^*))(x)=(\nn_x(\nn_x V)^*)(x).
\end{equation}
For $A=(A_1,A_2,\cdots,A_p)\in\R^p\otimes\R^{pd}$ with $A_{k}\in
\R^p\otimes\R^d$, $k=1,\cdots,p,$ and $B\in\R^d$, let's  define
$A\circ B\in\R^p\otimes\R^p$ by
\begin{equation*}
A\circ B=(A_1B,A_2B,\cdots,A_pB).
\end{equation*}
Set, for any $\theta\in\Theta$,
\begin{equation}\label{z3}
I(\theta):=\int_0^T(\nn_\theta
b)^*(X_s^0,\mathscr{L}_{X_s^0},\theta)\hat\si(X_s^0)(\nn_\theta
b)(X_s^0,\mathscr{L}_{X_s^0},\theta)\d s,
\end{equation}
 and, for any random variable $\zeta\in\C$ with
$\mathscr{L}_\zeta\in\mathcal {P}_2(\C)$,
\begin{equation}\label{s0}
\Upsilon(\zeta,\theta_0):=(\nn_\theta
b)^*(\zeta,\mathscr{L}_{\zeta},\theta_0)\hat\si(\zeta)\si(\zeta,\mathscr{L}_{\zeta}).
\end{equation}
Furthermore,  we set
\begin{equation}\label{z2}
\begin{split}
K(\theta):&=-2\int_0^T\Big\{(\nn_\theta^{(2)}
b^*)(X_s^0,\mathscr{L}_{X_s^0},\theta)\circ\Big(\hat\si(X_s^0)
\LL(X_s^0,\theta,\theta_0)\Big)\Big\}\d s,~~~~~\theta\in\Theta.
\end{split}
\end{equation}

Another main result in this paper is presented as below, which
reveals the asymptotic distribution of $\hat\theta_{n,\vv}.$

\begin{thm}\label{th2}
{\rm Let the assumptions of Theorem \ref{th1} hold and suppose
further that ({\bf A2}) and ({\bf A3}) hold and that $I(\cdot)$ and
$K(\cdot)$ defined in \eqref{z3} and \eqref{z2}, respectively, are
continuous. Then,
\begin{equation*}
\vv^{-1}(\hat\theta_{n,\vv}-\theta_0)\rightarrow
I^{-1}(\theta_0)\int_0^T\Upsilon(X_s^0,\theta_0)\d B(s)~~~~\mbox{ in
probability }
\end{equation*}
as $\vv\rightarrow0$ and $n\rightarrow\8$, where $I(\cdot) $ and
$\Upsilon(\cdot)$ are given in \eqref{z3} and \eqref{s0},
respectively.

}
\end{thm}

Now, we provide an example to demonstrate our main results.
\begin{exa}
{\rm Let
$\theta=(\theta^{(1)},\theta^{(2)})^*\in\Theta_0:=(c_1,c_2)\times(c_3,c_4)\subset\R^2$
for some $c_1<c_2$ and $c_3<c_4.$ For any $\vv\in(0,1)$, consider
the following scalar path-distribution dependent SDE
\begin{equation}\label{d1}
\d X^\vv(t)=\theta^{(1)}+\theta^{(2)}\int_\C
b_0(X^\vv_t,\zeta)\mathscr{L}_{X^\vv_t}(\d\zeta)+\vv(1+|X^\vv(t)|)\,
\d B(t),~~~t\in(0,T]
\end{equation}
with the initial value $X_0^\vv=\xi,$ where $\theta\in\Theta_0$ is
an unknown parameter with the true value
$\theta_0=(\theta^{(1)}_0,\theta^{(2)}_0)\in\Theta_0$,  and
$b_0:\C\times\C\rightarrow\R$ satisfy the global Lipschitz
condition, i.e.,
  there
exists a constant $K>0$ such that
\begin{equation}\label{d2}
|b_0(\zeta_1,\zeta_2)-b(\zeta_1',\zeta_2')|
\le
K\{|\zeta_1-\zeta_1'|+|\zeta_2-\zeta_2'|\},~~~~~\zeta_1,\zeta_2,\zeta_1',\zeta_2'\in\C.
\end{equation}
For any $\zeta\in\C$, $\mu\in\mathcal {P}_2(\C)$ and
$\theta=(\theta^{(1)},\theta^{(2)})^*$, set
\begin{equation*}
b(\zeta,\mu,\theta):=\theta^{(1)}+\theta^{(2)}\int_\C
b_0(\zeta,\zeta')\mu(\d\zeta') ~~\mbox{ and
}~~\si(\zeta,\mu):=1+|\zeta(0)|.
\end{equation*}
Then, \eqref{d1} can be reformulated as \eqref{eq1}. By a direct
calculation, it follows from \eqref{d2} that, for any
$\mu,\nu\in\mathcal {P}_2(\C)$ and $\zeta_1,\zeta_2\in\C$,
\begin{equation}\label{d3}
\begin{split}
|b(\zeta_1,\mu,\theta)-b(\zeta_2,\nu,\theta)|&=|\theta^{(2)}|\cdot\Big|\int_\C
b_0(\zeta_1,\zeta)\mu(\d\zeta)-\int_\C
b_0(\zeta_2,\zeta')\nu(\d\zeta')\Big|\\
&\le|\theta^{(2)}| \int_\C\int_\C| b_0(\zeta_1,\zeta)-
b_0(\zeta_2,\zeta')|\pi(\d\zeta,\d\zeta')\\
&\le K|\theta^{(2)}|\int_\C\int_\C
\{|\zeta_1-\zeta_2|+|\zeta-\zeta'|\}\pi(\d\zeta,\d\zeta')\\
&\le K(|c_3|\vee|c_4|)\{|\zeta_1-\zeta_2|+\mathbb{W}_1(\mu,\nu)\}\\
&\le K(|c_3|\vee|c_4|)\{|\zeta_1-\zeta_2|+\mathbb{W}_2(\mu,\nu)\},
\end{split}
\end{equation}
in which $\pi\in\mathcal {C}(\mu,\nu)$. On the other hand, for any
$x,y\in\R$ and $\mu,\nu\in\mathcal {P}_2(\R)$, one has
\begin{equation*}
|\si(x,\mu)-\si(y,\nu)|\le|x-y|.
\end{equation*}
Hence, the assumption ({\bf A1}) holds for \eqref{d1}. Next, for any
$x,y\in\R$ and $\mu,\nu\in\mathcal {P}_2(\R)$, we have
\begin{equation*}
|\si^{-2}(x,\mu)-\si^{-2}(y,\nu)|=\Big|\ff{1}{(1+|x|)^2}-\ff{1}{(1+|y|)^2}\Big|\le
4|x-y|.
\end{equation*}
So, ({\bf A2}) is fulfilled. Furthermore, observe that
\begin{equation}\label{d4}
(\nn_\theta b)(\zeta,\mu,\theta)=\Big(1,\int_\C
b_0(\zeta,\zeta')\mu(\d\zeta')\Big)^*~~~\mbox{ and
}~~~(\nn_\theta(\nn_\theta b))(\zeta,\mu,\theta)={\bf 0}_{2\times2},
\end{equation}
where ${\bf 0}_{2\times2}$ stands for the $2\times 2$-zero matrix.
Thus, \eqref{d3} yields that both ({\bf B1})  and ({\bf B2}) hold.
We further assume that the initial value is global Lipschitz, i.e.,
there exists an $L>0$ such that
\begin{equation*}
|\xi(t)-\xi(s)|\le L|t-s|,~~~~t,s\in[-r_0,0].
\end{equation*}
As a consequence, concerning \eqref{d1}, the assumptions ({\bf
A1})-({\bf A3}) and ({\bf B1})-({\bf B2}) hold, respectively.

The discrete-time EM scheme associated with \eqref{d1} is given by
\begin{equation}
Y^\vv(t_k)=Y^\vv(t_{k-1})+\Big(\theta^{(1)}+\theta^{(2)}\int_\C
b_0(\hat Y^\vv_{t_{k-1}},\zeta)\mathscr{L}_{\hat
Y^\vv_{t_{k-1}}}(\d\zeta)\Big)\dd+\vv(1+|Y^\vv(t_{k-1})|)\triangle
B_k,~~~k\ge1,
\end{equation}
with $Y^\vv(t)=X^\vv(t)=\xi(t), t\in[-r_0,0],$ where $(\hat
Y^\vv_{t_k})$ is defined as in \eqref{w2}. According to \eqref{eq2},
the contrast function admits the form below
\begin{equation*}
\Psi_{n,\vv}(\theta)=\vv^{-2}\delta^{-1}\sum_{k=1}^n\ff{1}{(1+|Y^\vv(t_{k-1})|)^2}\Big|Y^\vv(t_k)-Y^\vv(t_{k-1})-\Big(\theta^{(1)}+\theta^{(2)}\int_\C
b_0(\hat Y^\vv_{t_{k-1}},\zeta)\mathscr{L}_{\hat
Y^\vv_{t_{k-1}}}(\d\zeta)\Big)\dd\Big|^2.
\end{equation*}
Observe that
\begin{equation*}
\begin{split}
\ff{\partial}{\partial\theta^{(1)}}\Psi_{n,\vv}(\theta)&=-2\,\vv^{-2}\sum_{k=1}^n\ff{1}{(1+|Y^\vv(t_{k-1})|)^2}\Big\{Y^\vv(t_k)-Y^\vv(t_{k-1})\\
&\quad-\Big(\theta^{(1)}+\theta^{(2)}\int_\C b_0(\hat
Y^\vv_{t_{k-1}},\zeta)\mathscr{L}_{\hat
Y^\vv_{t_{k-1}}}(\d\zeta)\Big)\dd\Big\},
\end{split}
\end{equation*}
and
\begin{equation*}
\begin{split}
\ff{\partial}{\partial\theta^{(2)}}\Psi_{n,\vv}(\theta)&=-2\,\vv^{-2}\sum_{k=1}^n\ff{1}{(1+|Y^\vv(t_{k-1})|)^2}\Big\{Y^\vv(t_k)-Y^\vv(t_{k-1})\\
&\quad-\Big(\theta^{(1)}+\theta^{(2)}\int_\C b_0(\hat
Y^\vv_{t_{k-1}},\zeta)\mathscr{L}_{\hat
Y^\vv_{t_{k-1}}}(\d\zeta)\Big)\dd\Big\} \int_\C b_0(\hat
Y^\vv_{t_{k-1}},\zeta)\mathscr{L}_{\hat Y^\vv_{t_{k-1}}}(\d\zeta).
\end{split}
\end{equation*}
Subsequently, solving the equation below
\begin{equation*}
\ff{\partial}{\partial\theta^{(1)}}\Psi_{n,\vv}(\theta)=\ff{\partial}{\partial\theta^{(2)}}\Psi_{n,\vv}(\theta)=0,
\end{equation*}
we obtain the LSE
$\hat\theta_{n,\vv}=(\hat\theta_{n,\vv}^{(1)},\hat\theta_{n,\vv}^{(2)})^*$
of the unknown parameter
$\theta=(\theta^{(1)},\theta^{(2)})^*\in\Theta_0$ possesses the
formula
\begin{equation*}
\hat\theta_{n,\vv}^{(1)}=\ff{A_2A_5-A_3A_4}{\dd(A_1A_5-A_4^2)}~~~~~\mbox{
and }~~~~~
\hat\theta_{n,\vv}^{(2)}=\ff{A_1A_3-A_2A_4}{\dd(A_1A_5-A_4^2)},
\end{equation*}
where
\begin{equation*}
A_1:=\sum_{k=1}^n\ff{1}{(1+|Y^\vv(t_{k-1})|)^2},~~~~~~~~~~~~~~~~~~~~~~~~~~~~~~
~~A_2:=\sum_{k=1}^n\ff{Y^\vv(t_k)-Y^\vv(t_{k-1})}{(1+|Y^\vv(t_{k-1})|)^2},
\end{equation*}
\begin{equation*}
A_3:=\sum_{k=1}^n\ff{(Y^\vv(t_k)-Y^\vv(t_{k-1}))\int_\C b_0(\hat
Y^\vv_{t_{k-1}},\zeta)\mathscr{L}_{\hat
Y^\vv_{t_{k-1}}}(\d\zeta)}{(1+|Y^\vv(t_{k-1})|)^2},~~~A_4:=\sum_{k=1}^n\ff{\int_\C
b_0(\hat Y^\vv_{t_{k-1}},\zeta)\mathscr{L}_{\hat
Y^\vv_{t_{k-1}}}(\d\zeta)}{(1+|Y^\vv(t_{k-1})|)^2},
\end{equation*}
and
\begin{equation*}
A_5:=\sum_{k=1}^n\ff{\Big(\int_\C b_0(\hat
Y^\vv_{t_{k-1}},\zeta)\mathscr{L}_{\hat
Y^\vv_{t_{k-1}}}(\d\zeta)\Big)^2}{(1+|Y^\vv(t_{k-1})|)^2}.
\end{equation*}
In terms of Theorem \ref{th1}, $\hat\theta_{n,\vv}\rightarrow\theta$
in probability as $\vv\rightarrow0$ and $n\rightarrow\infty$. Next,
from \eqref{d4}, it follows that
\begin{equation*}
I(\theta_0)=\int_0^T\ff{1}{(1+|X_s^0|)^2}\left(\begin{array}{ccc}
  1 & b_0(X_s^0,X_s^0)\\
  b_0(X_s^0,X_s^0) & b_0(X_s^0,X_s^0)^2\\
  \end{array}
  \right)\d
s,
\end{equation*}
and, for $\zeta\in\C,$
\begin{equation*}
\int_0^T\Upsilon(X_s^0,\theta_0)\d
B(s)=\int_0^T\ff{1}{1+|X^0(s)|}\left(\begin{array}{c}
  1 \\
  b_0(X_s^0,X_s^0)\\
  \end{array}
  \right)\d B(s).
\end{equation*}
At last, according to Theorem \ref{th2}, we conclude that
\begin{equation*}
\vv^{-1}(\hat\theta_{n,\vv}-\theta_0)\rightarrow
I^{-1}(\theta_0)\int_0^T\Upsilon(X_s^0,\theta_0)\d B(s)~~~~\mbox{ in
probability }
\end{equation*}
as $\vv\rightarrow0$ and $n\rightarrow\8$ provided that $I(\cdot)$
is positive definite.

}
\end{exa}

Before we proceed to complete the proof of Theorem \ref{th2}, let's
prepare the lemmas below.

\begin{lem}\label{le1}
{\rm Assume that ({\bf A1})- ({\bf A3}) and  ({\bf B1})- ({\bf B2})
hold. Then,
\begin{equation}\label{s1}
\int_0^T\Upsilon(\hat Y_{\lfloor t/\dd\rfloor\dd}^\vv,\theta_0)\d
B(t)\longrightarrow\int_0^T\Upsilon(X_t^0,\theta_0)\d B(t)~~~~\mbox{
in probability }
\end{equation}
as $\vv\rightarrow0$ and $\dd\rightarrow0$. Moreover,
\begin{equation}\label{s2}
\vv^{-1}(\nn_\theta\Phi_{n,\vv})(\theta_0)\rightarrow
-2\int_0^T\Upsilon(X_s^0,\theta_0)\d B(s)~~~~\mbox{ in probability }
\end{equation}
whenever $\vv\rightarrow0$ and $\dd\rightarrow0$.

}
\end{lem}

\begin{proof}
We first claim that
\begin{equation}\label{s7}
\int_0^T\|\Upsilon(\hat Y_{\lfloor
t/\dd\rfloor\dd}^\vv,\theta_0)-\Upsilon(X_t^0,\theta_0)\|^2\d
t\rightarrow0~~~~\mbox{ in probability }
\end{equation}
as $\vv\rightarrow0$ and $\dd\rightarrow0.$  For any $\kk>0$ and
$\rho>0$, by the aid of \eqref{s7} and by  making use of
\cite[Theorem 2.6, P.63]{F98}, we have
\begin{equation*}
\begin{split}
&\P\Big(\Big|\int_0^T(\Upsilon(\hat Y_{\lfloor
t/\dd\rfloor\dd}^\vv,\theta_0)-\Upsilon(X_t^0,\theta_0))\d
B(t)\Big|\ge\kk\Big)\\
&\le\P\Big( \int_0^T\|\Upsilon(\hat Y_{\lfloor
t/\dd\rfloor\dd}^\vv,\theta_0)-\Upsilon(X_t^0,\theta_0)\|^2\d  t
\ge\kk^2\rho\Big)+\rho.
\end{split}
\end{equation*}
Thus, \eqref{s1} follows from \eqref{s7} and the arbitrariness of
$\rho.$ So, in what follows, it remains to show that \eqref{s7}
holds true. Observe that
\begin{equation*}
\begin{split}
&\Upsilon(\hat Y_{\lfloor
t/\dd\rfloor\dd}^\vv,\theta_0)-\Upsilon(X_t^0,\theta_0)\\
&=\{(\nn_\theta b)^*(\hat Y_{\lfloor
t/\dd\rfloor\dd}^\vv,\mathscr{L}_{\hat Y_{\lfloor
t/\dd\rfloor\dd}^\vv},\theta_0)-(\nn_\theta
b)^*(X_t^0,\mathscr{L}_{X_t^0},\theta_0)\}\hat\si(\hat Y_{\lfloor
t/\dd\rfloor\dd}^\vv)\si(\hat Y_{\lfloor
t/\dd\rfloor\dd}^\vv,\mathscr{L}_{\hat Y_{\lfloor
t/\dd\rfloor\dd}^\vv})\\
&\quad+(\nn_\theta
b)^*(X_t^0,\mathscr{L}_{X_t^0},\theta_0)\{\hat\si(\hat Y_{\lfloor
t/\dd\rfloor\dd}^\vv)-\hat\si(X_t^0)\}\si(\hat Y_{\lfloor
t/\dd\rfloor\dd}^\vv,\mathscr{L}_{\hat Y_{\lfloor
t/\dd\rfloor\dd}^\vv})\\
&\quad+(\nn_\theta
b)^*(X_t^0,\mathscr{L}_{X_t^0},\theta_0)\hat\si(X_t^0)\{\si(\hat
Y_{\lfloor t/\dd\rfloor\dd}^\vv,\mathscr{L}_{\hat Y_{\lfloor
t/\dd\rfloor\dd}^\vv})-\si(X_t^0,\mathscr{L}_{X_t^0})\\
&=:\Sigma_1(t,\vv,\dd)+\Sigma_2(t,\vv,\dd)+\Sigma_3(t,\vv,\dd).
\end{split}
\end{equation*}
From ({\bf B1}),   \eqref{r3}, and \eqref{a6},  it follows that
\begin{equation}\label{s8}
\begin{split}
&\int_0^T(\|\Sigma_1(t,\vv,\dd)\|^2+\|\Sigma_2(t,\vv,\dd)\|^2)\d t\\
&\le c\int_0^T(1+\|\hat Y_{\lfloor
t/\dd\rfloor\dd}^\vv\|_\8^4)\|\hat Y_{\lfloor
t/\dd\rfloor\dd}^\vv-X_t^0\|_\8^2\d t + \hat\Pi(\vv,\dd)\\
&\le c\int_0^T(1+\|\hat Y_{\lfloor
t/\dd\rfloor\dd}^\vv-X_t^0\|_\8^4+\|X_t^0\|_\8^4)\|\hat Y_{\lfloor
t/\dd\rfloor\dd}^\vv-X_t^0\|_\8^2\d s + \hat\Pi(\vv,\dd)\\
&\le c\int_0^T(1+\|\hat Y_{\lfloor
t/\dd\rfloor\dd}^\vv-X_t^0\|_\8^4)\|\hat Y_{\lfloor
t/\dd\rfloor\dd}^\vv-X_t^0\|_\8^2\d t + \hat\Pi(\vv,\dd),
\end{split}
\end{equation}
where
\begin{equation*}
\hat\Pi(\vv,\dd):=c\int_0^T(1+
 \|\hat Y_{\lfloor
t/\dd\rfloor\dd}^\vv\|_\8^4)\E\|\hat Y_{\lfloor
t/\dd\rfloor\dd}^\vv-X_t^0\|^2_\8\d t.
\end{equation*}
For any $\rho>0$, one gets from \eqref{s8} that
\begin{equation*}
\begin{split}
&\P\Big(\int_0^T(\|\Sigma_1(t,\vv,\dd)\|^2+\|\Sigma_2(t,\vv,\dd)\|^2)\d
t\ge\rho\Big)\\
&\le\P( \hat\Pi(\vv,\dd)\ge\rho/2)+\P\Big(c\int_0^T(1+\|\hat
Y_{\lfloor t/\dd\rfloor\dd}^\vv-X_t^0\|_\8^4)\|\hat Y_{\lfloor
t/\dd\rfloor\dd}^\vv-X_t^0\|_\8^2\d t\ge\ff{\rho}{2}\Big).
\end{split}
\end{equation*}
By the Chebyshev inequality, in addition to \eqref{r4} and
\eqref{y2},
\begin{equation*}
\begin{split}
\P( \hat\Pi(\vv,\dd)\ge\rho/2)&\le\ff{c}{\rho}\int_0^T(1+\E
 \|\hat Y_{\lfloor
t/\dd\rfloor\dd}^\vv\|_\8^4 )\E\|\hat Y_{\lfloor
t/\dd\rfloor\dd}^\vv-X_t^0\|^2_\8\d s\\
&\longrightarrow0
\end{split}
\end{equation*}
as $\vv\rightarrow0$ and $\dd\rightarrow0.$ Also, for any $K>0$, by
Chebyshev's inequality, besides \eqref{r4},
\begin{align*}
&\P\Big(c\int_0^T(1+\|\hat Y_{\lfloor
t/\dd\rfloor\dd}^\vv-X_t^0\|_\8^4)\|\hat Y_{\lfloor
t/\dd\rfloor\dd}^\vv-X_t^0\|_\8^2\d t\ge\ff{\rho}{2}\Big)\\
&\le\P\Big(c(1+K^4)\int_0^T\|\hat Y_{\lfloor
t/\dd\rfloor\dd}^\vv-X_t^0\|_\8^2\d t\ge\ff{\rho}{4}\Big)\\
&\quad+\P\Big(c\int_0^T(1+\|\hat Y_{\lfloor
t/\dd\rfloor\dd}^\vv\|_\8^8){\bf1}_{\{\|\hat Y_{\lfloor
t/\dd\rfloor\dd}^\vv-X_t^0\|_\8\ge K\}}\d t\ge\ff{\rho}{4}\Big)\\
&\le \ff{c(1+K^4)}{\rho}\int_0^T\E\|\hat Y_{\lfloor
t/\dd\rfloor\dd}^\vv-X_t^0\|_\8^2\d t\\
&\quad+\ff{c}{\rho}\int_0^T\E((1+\|\hat Y_{\lfloor
t/\dd\rfloor\dd}^\vv\|_\8^8){\bf1}_{\{\|\hat Y_{\lfloor
t/\dd\rfloor\dd}^\vv-X_t^0\|_\8\ge K\}})\d t\\
&\le \ff{c(1+K^4)}{\rho}\int_0^T\E\|\hat Y_{\lfloor
t/\dd\rfloor\dd}^\vv-X_t^0\|_\8^2\d t\\
&\quad+\ff{c}{\rho}\int_0^T\Big(1+\E\|\hat Y_{\lfloor
t/\dd\rfloor\dd}^\vv\|_\8^{16}\Big)^{1/2}\Big(\P(\|\hat Y_{\lfloor
t/\dd\rfloor\dd}^\vv-X_t^0\|_\8\ge K)\Big)^{1/2}\d t\\
&\le\ff{c(1+K^4)}{\rho}\int_0^T\E\|\hat Y_{\lfloor
t/\dd\rfloor\dd}^\vv-X_t^0\|_\8^2\d t+\ff{c}{\rho
K}\int_0^T\Big(\E\|\hat Y_{\lfloor
t/\dd\rfloor\dd}^\vv-X_t^0\|_\8^2\Big)^{1/2}\d t.
\end{align*}
This, together with \eqref{y2}, leads  to
\begin{equation}\label{a2}
\int_0^T(\|\Sigma_1(t,\vv,\dd)\|^2+\|\Sigma_2(t,\vv,\dd)\|^2)\d
t\longrightarrow0~~~\mbox{ in probability }
\end{equation}
as $\vv\rightarrow0$ and $\dd\rightarrow0$. Furthermore, ({\bf A1}),
\eqref{a6} as well as  ({\bf B1}) imply that
\begin{equation}\label{a3}
\int_0^T \E\|\Sigma_3(t,\vv,\dd)\|^2 \d t\le c\int_0^T\E\|\hat
Y_{\lfloor t/\dd\rfloor\dd}^\vv-X_t^0\|_\8^2\d t\longrightarrow0
\end{equation}
as $\vv\rightarrow0$ and $\dd\rightarrow0$. As a result, \eqref{s7}
follows from \eqref{a2}, \eqref{a3} and Chebyshev's inequality.

For any $\theta\in\Theta$ and random variable $\zeta\in\C$ with
$\mathcal {P}_2(\C)$, note from \eqref{c1} that
\begin{equation*}
(\nn_{\theta}\Lambda)(\zeta,\theta,\theta_0)=-(\nn_\theta
b)(\zeta,\mathscr{L}_\zeta,\theta).
\end{equation*}
A straightforward calculation shows that
\begin{equation*}
\begin{split}
(\nn_\theta\Phi_{n,\vv})(\theta)&=2\sum_{k=1}^n(\nn_\theta\Lambda)^*(\hat
Y_{t_{k-1}}^\vv,\theta,\theta_0)\hat\si(\hat
Y_{t_{k-1}}^\vv)\Big\{P_k(\theta_0)
+\dd\Lambda(\hat Y_{t_{k-1}}^\vv,\theta,\theta_0)\Big\}\\
&=-2\sum_{k=1}^n(\nn_\theta b)^*(\hat
Y_{t_{k-1}}^\vv,\mathscr{L}_{\hat
Y_{t_{k-1}}^\vv},\theta)\hat\si(\hat Y_{t_{k-1}}^\vv)P_k(\theta).
\end{split}
\end{equation*}
Therefore, one has
\begin{equation*}
\begin{split}
\vv^{-1}(\nn_\theta\Phi_{n,\vv})(\theta_0) &=-2\int_0^T\Upsilon(\hat
Y_{\lfloor s/\dd\rfloor\dd}^\vv,\theta_0)\d B(s).
\end{split}
\end{equation*}
Subsequently, \eqref{s2} follows from \eqref{s1} immediately.
\end{proof}

\begin{lem}\label{lem3.3}
{\rm Under the assumptions of Theorem \ref{th2},
\begin{equation}\label{a11}
(\nn_\theta^{(2)}\Phi_{n,\vv})(\theta)\longrightarrow
K_0(\theta):=K(\theta)+2I(\theta)~~~\mbox{ in probability }
\end{equation}
as $\vv\rightarrow0, n\rightarrow\8$, where
$(\nn_\theta^{(2)}\Phi_{n,\vv}), I(\theta), K(\theta)$ are defined
as in \eqref{z1}, \eqref{z3}, and \eqref{z2}, respectively. }
\end{lem}

\begin{proof}
By the chain rule, we infer from \eqref{z1} that
\begin{align*}
(\nn_\theta^{(2)}\Phi_{n,\vv})(\theta)&=-2\sum_{k=1}^n(\nn_\theta^{(2)}
b^*)(\hat Y_{t_{k-1}}^\vv,\mathscr{L}_{\hat Y_{t_{k-1}}^\vv},\theta)\circ\Big(\hat\si(\hat Y_{t_{k-1}}^\vv)P_k(\theta)\Big)\\
&\quad-2\sum_{k=1}^n(\nn_\theta b)^*(\hat
Y_{t_{k-1}}^\vv,\mathscr{L}_{\hat
Y_{t_{k-1}}^\vv},\theta)\hat\si(\hat Y_{t_{k-1}}^\vv)(\nn_\theta
P_k)(\theta)\\
&=-2\sum_{k=1}^n(\nn_\theta^{(2)}
b^*)(\hat Y_{t_{k-1}}^\vv,\mathscr{L}_{\hat Y_{t_{k-1}}^\vv},\theta)\circ\Big(\hat\si(\hat Y_{t_{k-1}}^\vv)P_k(\theta_0)\Big)\\
&\quad-2\,\dd\sum_{k=1}^n\Big\{(\nn_\theta^{(2)}
b^*)(\hat Y_{t_{k-1}}^\vv,\mathscr{L}_{\hat Y_{t_{k-1}}^\vv},\theta)\circ\Big(\hat\si(\hat Y_{t_{k-1}}^\vv)\Lambda(\hat Y_{t_{k-1}}^\vv,\theta,\theta_0)\Big)\\
&\quad-(\nn_\theta b)^*(\hat Y_{t_{k-1}}^\vv,\mathscr{L}_{\hat
Y_{t_{k-1}}^\vv},\theta)\hat\si(\hat Y_{t_{k-1}}^\vv)(\nn_\theta
b)(\hat Y_{t_{k-1}}^\vv,\mathscr{L}_{\hat Y_{t_{k-1}}^\vv},\theta)\Big\}\\
&=:\Theta_1(\vv,\dd)+\Theta_2(\vv,\dd).
\end{align*}
Taking ({\bf B2}) into consideration and mimicking the argument of
Lemma \ref{lem4}, we obtain that
\begin{equation*}
\Theta_1(\vv,\dd)\rightarrow0~~~\mbox{ in probability as
}\vv\rightarrow0,~~\dd\rightarrow0.
\end{equation*}
Observe that
\begin{equation*}
\begin{split}
\Theta_2(\vv,\dd)&=-2\int_0^T\Big\{(\nn_\theta^{(2)} b^*)(\hat
Y_{\lfloor s/\dd\rfloor\dd}^\vv,\mathscr{L}_{\hat Y_{\lfloor
s/\dd\rfloor\dd}^\vv},\theta) \circ\Big(\hat\si(\hat Y_{\lfloor
s/\dd\rfloor\dd}^\vv) \Lambda(\hat Y_{\lfloor
s/\dd\rfloor\dd}^\vv,\theta,\theta_0) \Big)\d s\\
&\quad+ 2\int_0^T\Big\{(\nn_\theta b)^*(\hat Y_{\lfloor
s/\dd\rfloor\dd}^\vv,\mathscr{L}_{\hat Y_{\lfloor
s/\dd\rfloor\dd}^\vv},\theta)\hat\si(\hat Y_{\lfloor
s/\dd\rfloor\dd}^\vv)(\nn_\theta b)(\hat Y_{\lfloor
s/\dd\rfloor\dd}^\vv,\mathscr{L}_{\hat Y_{\lfloor
s/\dd\rfloor\dd}^\vv},\theta)\d s\\
&=:\Psi_1(\vv,\dd)+\Psi_2(\vv,\dd).
\end{split}
\end{equation*}
Carrying out an analogous argument to derive Lemma \ref{lem2}, we
infer that
\begin{equation}\label{z4}
\Psi_1(\vv,\dd)\rightarrow K(\theta)~~~\mbox{ in probability }
\mbox{ as } \varepsilon\rightarrow0, \delta\rightarrow0
\end{equation}
by taking ({\bf B2}) into account, and that
\begin{equation}\label{z5}
\Psi_2(\vv,\dd)\rightarrow 2I(\theta)~~~\mbox{ in probability }
\mbox{ as } \varepsilon\rightarrow0, \delta\rightarrow0
\end{equation}
by using ({\bf B1}). Thus, the desired assertion follows from
\eqref{z4} and \eqref{z5} immediately.
\end{proof}

Now we start to finish the argument of Theorem \ref{th2} on the
basis of the previous lemmas.

\begin{proof}[{\bf Proof of Theorem \ref{th2}}]
The original idea on the proof of Theorem \ref{th2} is taken from
\cite{U04}. To make the content self-contained, we herein provide a
sketch of the proof. In terms of Theorem \ref{th1}, there exists a
sequence $\eta_{n,\vv}\rightarrow0$ as $\vv\rightarrow0$ and
$n\rightarrow\8$ such that $\hat\theta_{n,\vv}\in
B_{\eta_{n,\vv}}(\theta_0)\subset\Theta$, $\P$-a.s.  By the Taylor
expansion, one has
\begin{equation}\label{a4}
(\nn_\theta\Phi_{n,\vv})(\hat\theta_{n,\vv})=(\nn_\theta\Phi_{n,\vv})(\theta_0)+D_{n,\vv}(\hat\theta_{n,\vv}-\theta_0),~~~\hat\theta_{n,\vv}\in
B_{\eta_{n,\vv}}(\theta_0)
\end{equation}
with
\begin{equation*}
D_{n,\vv}:=\int_0^1(\nn_\theta^{(2)}\Phi_{n,\vv})(\theta_0
+u(\hat\theta_{n,\vv}-\theta_0))\d u,~~~~\hat\theta_{n,\vv}\in
B_{\eta_{n,\vv}}(\theta_0).
\end{equation*}
 Observe that, for $\hat\theta_{n,\vv}\in
B_{\eta_{n,\vv}}(\theta_0)$,
\begin{equation*}
\begin{split}
\|D_{n,\vv}-K_0(\theta_0)\|&\le\|D_{n,\vv}-(\nn_\theta^{(2)}\Phi_{n,\vv})(\theta_0)\|+\|(\nn_\theta^{(2)}\Phi_{n,\vv})(\theta_0)-K_0(\theta_0)\|\\
&\le\int_0^1\|(\nn_\theta^{(2)}\Phi_{n,\vv})(\theta_0
+u(\hat\theta_{n,\vv}-\theta_0))-(\nn_\theta^{(2)}\Phi_{n,\vv})(\theta_0)\|\d
u\\
&\quad+ \|(\nn_\theta^{(2)}\Phi_{n,\vv})(\theta_0)-K_0(\theta_0)\|\\
&\le \sup_{\theta\in
B_{\eta_{n,\vv}}(\theta_0)}\|(\nn_\theta^{(2)}\Phi_{n,\vv})(\theta)-(\nn_\theta^{(2)}\Phi_{n,\vv})(\theta_0)\|+
\|(\nn_\theta^{(2)}\Phi_{n,\vv})(\theta_0)-K_0(\theta_0)\|\\
&\le\sup_{\theta\in
B_{\eta_{n,\vv}}(\theta_0)}\|(\nn_\theta^{(2)}\Phi_{n,\vv})(\theta)-K_0(\theta)\|+\sup_{\theta\in
B_{\eta_{n,\vv}}(\theta_0)}\|K_0(\theta)-K_0(\theta_0)\|\\
&\quad+2\|(\nn_\theta^{(2)}\Phi_{n,\vv})(\theta_0)-K_0(\theta_0)\|,
\end{split}
\end{equation*}
in which $K_0(\cdot)$ is introduced in \eqref{a11}. This, together
with Lemma \ref{lem3.3} and  continuity of $K_0(\cdot)$, gives that
\begin{equation}\label{a0}
D_{n,\vv}\rightarrow K_0(\theta_0)~~~~\mbox{ in probability }
\end{equation}
as $\vv\rightarrow0$ and $n\rightarrow\8.$ By following the exact
line of \cite[Theorem 2.2]{LSS}, we can deduce that $D_{n,\vv}$ is
invertible on the set
\begin{equation*}
\Gamma_{n,\vv}:=\Big\{\sup_{\theta\in
B_{\eta_{n,\vv}}(\theta_0)}\|(\nn_\theta^{(2)}\Phi_{n,\vv})(\theta)-K_0(\theta_0)\|\le\ff{\aa}{2},~~\hat\theta_{n,\vv}\in
B_{\eta_{n,\vv}}(\theta_0) \Big\}
\end{equation*}
for some constant $\aa>0.$  Let
\begin{equation*}
\mathscr{D}_{n,\vv}=\{D_{n,\vv} \mbox{ is invertible },
\hat\theta_{n,\vv}\in B_{\eta_{n,\vv}}(\theta_0) \}.
\end{equation*}
By virtue of Lemma \ref{lem3.3}, one has
\begin{equation}\label{n1}
\lim_{\vv\rightarrow0,n\rightarrow\8}\P\Big(\sup_{\theta\in
B_{\eta_{n,\vv}}(\theta_0)}\|(\nn_\theta^{(2)}\Phi_{n,\vv})(\theta)-K_0(\theta_0)\|\le\ff{\aa}{2}\Big)=1.
\end{equation}
On the other hand, recall that
\begin{equation}\label{n2}
\lim_{\vv\rightarrow0,n\rightarrow\8}\P\Big(\hat\theta_{n,\vv}\in
B_{\eta_{n,\vv}}(\theta_0)\Big)=1.
\end{equation}
By the fundamental fact: for any events $A,B$,
$\P(AB)=\P(A)+\P(B)-\P(A\cup B)$, we observe that
\begin{equation}\label{n3}
\begin{split}
1\ge\P(\Gamma_{n,\vv})&\ge\P\Big(\sup_{\theta\in
B_{\eta_{n,\vv}}(\theta_0)}\|(\nn_\theta^{(2)}\Phi_{n,\vv})(\theta)-K_0(\theta_0)\|\le\ff{\aa}{2}\Big)\\
&\quad+\P\Big(\hat\theta_{n,\vv}\in
B_{\eta_{n,\vv}}(\theta_0)\Big)-1.
\end{split}
\end{equation}
Thus, taking advantage of \eqref{n1}, \eqref{n2} as well as
\eqref{n3}, we deduce from Sandwich theorem that
\begin{equation}\label{n4}
\P(\mathscr{D}_{n,\vv})\ge \P(\Gamma_{n,\vv})\rightarrow1
\end{equation}
 as $\vv\rightarrow0$ and $n\rightarrow\8$. Set
\begin{equation*}
U_{n,\vv}:=D_{n,\vv}{\bf1}_{\mathscr{D}_{n,\vv}}+I_{p\times
p}{\bf1}_{\mathscr{D}_{n,\vv}^c},
\end{equation*}
where $I_{p\times p}$ is a $p\times p$ identity matrix. For
$S_{n,\vv}:=\vv^{-1}(\hat\theta_{n,\vv}-\theta_0)$, we deduce from
\eqref{a4}  that
\begin{align*}
S_{n,\vv}&=S_{n,\vv}{\bf1}_{\mathscr{D}_{n,\vv}}+ S_{n,\vv}{\bf1}_{\mathscr{D}_{n,\vv}^c}\\
&=U_{n,\vv}^{-1}D_{n,\vv}S_{n,\vv}{\bf1}_{\mathscr{D}_{n,\vv}}+ S_{n,\vv}{\bf1}_{\mathscr{D}_{n,\vv}^c}\\
&=\vv^{-1}U_{n,\vv}^{-1}\{(\nn_\theta\Phi_{n,\vv})(\hat\theta_{n,\vv})-(\nn_\theta\Phi_{n,\vv})(\theta_0)\}{\bf1}_{\mathscr{D}_{n,\vv}}
+ S_{n,\vv}{\bf1}_{\mathscr{D}_{n,\vv}^c}\\
&=-\vv^{-1}U_{n,\vv}^{-1}(\nn_\theta\Phi_{n,\vv})(\theta_0){\bf1}_{\mathscr{D}_{n,\vv}}
+ S_{n,\vv}{\bf1}_{\mathscr{D}_{n,\vv}^c}\\
&\rightarrow I^{-1}(\theta_0)\int_0^T\Upsilon(X_s^0,\theta_0)\d
B(s),
\end{align*}
as $\vv\rightarrow0$ and $n\rightarrow\infty$, where in the forth
identity we dropped the term
$(\nn_\theta\Phi_{n,\vv})(\hat\theta_{n,\vv})$ according to the
notion of LSE and Fermat's lemma, and the last display follows from
Lemma \ref{le1}, \eqref{a0} as well as \eqref{n4} and by noting
$K_0(\theta_0)=2I(\theta_0)$. We therefore complete the proof.
\end{proof}

\end{document}